\documentclass[12 pt]{amsart}
\usepackage{amssymb}

\newcommand{\N}{\mathbb N}

\newcommand{\R}{\mathbb R}

\usepackage[margin=2.5cm]{geometry}
\usepackage{amsmath,amsthm,amssymb,amscd}
\usepackage[colorlinks=true, pdfstartview=FitV, linkcolor=blue,
citecolor=red,urlcolor=blue]{hyperref}

\newtheorem{thm}{Theorem}[section]
\newtheorem{corollaire}[thm]{Corollary}
\newtheorem{lem}[thm]{Lemma}
\newtheorem{hyp}[thm]{Assumption}
\newtheorem{prop}[thm]{Proposition}
\theoremstyle{definition}
\newtheorem{definition}[thm]{Definition}
\newtheorem{rem}[thm]{Remark}

\numberwithin{equation}{section}

\def\be#1 {\begin{equation} \label{#1}}
\newcommand{\ee}{\end{equation}}

\def \virg {\, , \,\,}

\newcommand{\BMO}{\textrm{BMO}}
\newcommand\<{\langle}
\renewcommand\>{\rangle}
\let\Re=\relax
\DeclareMathOperator{\Re}{Re}

\def\Xint#1{\mathchoice
   {\XXint\displaystyle\textstyle{#1}}%
   {\XXint\textstyle\scriptstyle{#1}}%
   {\XXint\scriptstyle\scriptscriptstyle{#1}}%
   {\XXint\scriptscriptstyle\scriptscriptstyle{#1}}%
   \!\int}
\def\XXint#1#2#3{{\setbox0=\hbox{$#1{#2#3}{\int}$}
     \vcenter{\hbox{$#2#3$}}\kern-.5\wd0}}

\def\aver#1{\Xint-_{#1}}

\newcommand{\calM}{\ensuremath{\mathcal{M}}}
\newcommand{\diam}{{\mathrm{diam}}}

\title[Dispersion via the heat semigroup]{Dispersive estimates with loss of
derivatives via the heat semigroup and the wave operator}

\author{Fr\'ed\'eric Bernicot}
\address{Fr\'ed\'eric Bernicot - CNRS - Universit\'e de Nantes \\ Laboratoire Jean
Leray \\ 2, rue de la Houssini\`ere
44322 Nantes cedex 3, France}
\email{frederic.bernicot@univ-nantes.fr }
\urladdr{http://www.math.sciences.univ-nantes.fr/~bernicot}

\author{Valentin Samoyeau}
\address{Valentin Samoyeau - Universit\'e de Nantes \\ Laboratoire Jean Leray
\\ 2, rue de la Houssini\`ere
44322 Nantes cedex 3, France}
\email{valentin.samoyeau@univ-nantes.fr}

\thanks{F. Bernicot's research is supported by ANR projects AFoMEN no.
2011-JS01-001-01 and HAB no. ANR-12-BS01-0013.
V. Samoyeau's research is supported by Centre Henri Lebesgue (program
"Investissements d'avenir" --- ANR-11-LABX-0020-01).}

\date{July 15, 2014}

\numberwithin{equation}{section}

\begin{document}

\begin{abstract} 
This paper aims to give a general (possibly compact or noncompact) analog of Strichartz 
inequalities with loss of derivatives, obtained by Burq, G\'erard, and Tzvetkov \cite{BGT} 
and Staffilani and Tataru \cite{ST}. Moreover we present a new approach, relying only on 
the heat semigroup in order to understand the analytic connexion between the heat semigroup 
and  the unitary Schr\"odinger group (both related to a same self-adjoint operator). One of 
the novelty is to forget the endpoint $L^1-L^\infty$ dispersive estimates and to look for a 
weaker $H^1-\BMO$ estimates (Hardy and BMO spaces both adapted to the heat semigroup).
This new point of view allows us to give a general framework (infinite metric spaces, 
Riemannian manifolds with rough metric, manifolds with boundary, \dots) where Strichartz 
inequalities with loss of derivatives can be reduced to microlocalized $L^2-L^2$ dispersive properties. 
We also use the link between the wave propagator and the unitary Schr\"odinger group to 
prove how short time dispersion for waves implies dispersion for the Schr\"odinger group.
\end{abstract}



\maketitle

\begin{quote}
\footnotesize\tableofcontents
\end{quote}

\section{Introduction}\label{section_introduction}

A powerful tool to study nonlinear Schr\"odinger equations is the family of so called Strichartz estimates. 
Those estimates are useful to control the size of solutions to a linear problem in terms of the size of the 
initial data. The ``size'' notion is usually given by a suitable functional space $L^p_tL^q_x$.
Such inequalities were first introduced by Strichartz in \cite{Strichartz} for Schr\"odinger waves on 
Euclidean space. They were then extended by Ginibre and Velo in \cite{GV} (and the endpoint is due to 
Keel and Tao in \cite{KT}) for the propagator operator
associated with the linear Schr\"odinger equation in $\mathbb R^d$. So for an initial data $u_0$, we 
are interested in controlling $u(t,\ldotp) = e^{it\Delta}u_0$ which is the solution of the linear Schr\"odinger equation:
$$\begin{cases}
   &i\partial_tu + \Delta u=0 \\
   &u_{|t=0}=u_0
  \end{cases}.$$
It is well-known that the unitary group $e^{it\Delta}$ satisfies the following inequality: $$\|e^{it\Delta}u_0\|_{L^pL^q([-T,T] \times \mathbb R^d)} \leq C_T \|u_0\|_{L^2(\mathbb R^d)}$$
for every pair $(p,q)$ of admissible exponents which means : $2\leq p,q \leq \infty$, $(p,q,d)\neq(2,\infty,2)$, and 
\begin{equation}\label{admissible}\tag{$p,q$}
 \frac 2p + \frac dq = \frac d2.
\end{equation}
The Strichartz estimates can be deduced via a $TT^*$ argument from the dispersive estimates 
\begin{equation}\label{disp_L1_Linfty}
 \|e^{it \Delta}u_0\|_{L^{\infty}( \mathbb R^d)} \lesssim |t|^{-\frac d2} \|u_0\|_{L^1(\mathbb R^d)}.
\end{equation}
If $\sup_{T>0}C_T<+\infty$, we will say that a global-in-time Strichartz estimate holds. Such a global-in-time 
estimate has been proved by Strichartz for the flat Laplacian on $\mathbb R^d$ 
while the local-in-time estimate is known in several geometric situation where the manifold is nontrapping 
(asymptotically Euclidean, conic, or hyperbolic, Heisenberg group); see \cite{BT,Bouclet, HTW,ST,BGX,AP} or 
with variable coefficients \cite{RZ,T}. 
The finite volume of the manifold and the presence of trapped geodesics appear to limit the extent to 
which dispersion can occur, see \cite{BGH}. 

The situation for compact manifolds presents a new difficulty, since considering the
constant initial data $u_0= 1\in L^2$ yields a contradiction for large time.
%

Burq, G\'erard, and Tzvetkov \cite{BGT} and Staffilani and Tataru \cite{ST} proved that Strichartz estimates
hold on compact manifolds for finite time if one considers regular data $u_0\in W^{1/p,2}$. Those are 
called ``with a loss of derivatives''. An interesting problem is to determine for specific situations, 
which loss of derivatives is optimal (for example the work of Bourgain \cite{Bourgain} on the flat torus and \cite{TT} of Takaoka and Tzvetkov).

Numerous recent works aim also to obtain such Strichartz estimates with a loss of 
derivatives in various situations, for example corresponding to a Laplacian operator 
on a smooth domain with boundary condition (Dirichlet or Neumann) see the works of 
Anton \cite{Ramona}, Blair-Smith-Sogge \cite{BSS} and Blair-Ford-Herr-Marzuola 
\cite{BFHM}. All these works are built on the approach for compact manifolds of 
\cite{BGT}. Concerning noncompact manifolds, Strichartz estimates with the same loss 
of derivatives have been obtained in \cite{BGT2} by Burq-G\'erard-Tzvetkov for the 
complement of a smooth and bounded domain in the Euclidean space. Let us precise 
that the two approaches \cite{BGT} (for the compact situation) and \cite{BGT2} (for
the non-compact situation) are completely different, although they give exactly the
same loss of derivatives. Indeed in \cite{BGT} the loss of derivatives is due to the
use of the only semi-classical dispersive inequality and in \cite{BGT2} the loss of
derivatives is due to the use of Sobolev embeddings together with the local
smoothing near the boundary. The case of infinite manifolds (with boundary) with one
trapped orbit was considered by Christianson in \cite{Christianson} where a larger
loss of derivatives of $1/p+\varepsilon$ is obtained. There the author allows to
perturb the Laplacian by a smooth potential.

We remark that, by Sobolev embedding, the loss of $2/p$ derivatives is straightforward.
Indeed $W^{\frac 2p,2} \hookrightarrow L^p$ since $d(\frac 12-\frac 1q)=\frac 2p$ so that 
\begin{equation} \|e^{it\Delta}u_0\|_{L^p} \lesssim \|e^{it\Delta}u_0\|_{W^{\frac 2p,2}} \leq \|u_0\|_{W^{\frac 2p,2}} \label{eq:tri} \end{equation}
and taking the $L^q([-T,T])$ norm it comes $$\|e^{it\Delta}u_0\|_{L^pL^q} \leq C_T \|u_0\|_{W^{\frac 2p,2}}.$$

So Strichartz estimates with loss of derivatives are interesting for a loss, smaller than $2/p$.

\bigskip

The purpose of this article is multiple:
\begin{itemize}
 \item To present a general/unified result with loss of derivatives for a (possibly compact or noncompact) general setting
(involving metric spaces with a self-adjoint generator);
 \item To try to understand the link between the heat semigroup and the unitary Schr\"odinger group, through the use 
of corresponding Hardy and BMO spaces. Such spaces allow us to get around the pointwise dispersive estimates and only to consider $L^2-L^2$ localized estimates (in space and in frequency); 
 \item To connect (short time) dispersive properties for Schr\"odinger group with dispersion for the wave propagators.
\end{itemize}

\bigskip

Let us set the general framework of our study.
Let $(X, d, \mu)$ be a metric measured space of homogeneous type. 
That is $d$ is a metric on $X$ and $\mu$ a nonnegative $\sigma$-finite Borel measure satisfying the doubling property: 
$$\forall x \in X, \forall r>0, \, \mu(B(x,2r)) \lesssim \mu(B(x,r)),$$
where $B(x,r)$ denote the open ball with center $x \in X$ and radius $r>0$.
As a consequence, there exists a homogeneous dimension $d>0$, such that
\begin{equation} \label{dd} \forall x \in X \virg \forall r>0 \virg \forall t \geq 1 \virg \mu(B(x,tr))\lesssim t^d \mu(B(x,r)). \end{equation}
Thus we aim our result to apply in numerous cases of metric spaces such as open subsets of $\mathbb R^d$, smooth $d$-manifolds, some fractal sets, Lie groups, Heisenberg group, \dots

Keeping in mind the canonical example of the Laplacian operator in $\mathbb R^d$: $\Delta = \sum_{1\leq j \leq d} \partial_j^2$, we will be more general in the following sense:
we consider a nonnegative, self-adjoint operator $H$ on $L^2=L^2(X,\mu)$ densely defined, which means that its domain 
$$ {\mathcal D}(H):=\{f \in L^2,\, Hf \in L^2\}$$
is supposed to be dense in $L^2$. It is known that $-H$ is the generator of a $L^2$-holomorphic semigroup $(e^{-tH})_{t\geq0}$ 
(see Definition \ref{def_semi-groupe} and \cite{Davies}) and we assume that it satisfies $L^2$ Davies-Gaffney estimates: 
for every $t>0$ and every subsets $E,F \subset X$
\begin{equation}\label{davies_gaffney_estimates}\tag{$DG$}
\|e^{-tH}\|_{L^2(E) \to L^2(F)} \lesssim e^{- \frac{d(E,F)^2}{4t}}
\end{equation}
(with the restriction to $t\lesssim \diam(X)$ if $X$ is bounded).
Without losing generality (up to consider $\lambda H$ for some positive real $\lambda>0$), we assumed 
that $H$ satisfies the previous normalized estimates, which are equivalent to a finite speed propagation 
property at speed $1$ (see later \eqref{speed_propagation}).

We will assume also that the heat semigroup $(e^{-tH})_{t\geq0}$ satisfies the typical upper estimates 
(for a second order operator): that for every $t>0$  the operator $e^{-tH}$ admits a kernel $p_t$ with
\begin{equation}\label{due} \tag{$DU\!E$}
0\leq  p_{t}(x,x)\lesssim
\frac{1}{\mu(B(x,\sqrt{t}))}, \quad \forall~t>0,\,\mbox{a.e. }x\in X.
\end{equation}

It is well-known that such on-diagonal pointwise estimates self-improve into the 
full pointwise Gaussian estimates  (see \cite[Theorem 1.1]{Gr1} or \cite[Section 4.2]{CS} e.g.):
\begin{equation} \label{UE} \tag{$U\!E$}
0\leq p_{t}(x,y)\lesssim
\frac{1}{\mu(B(x,\sqrt{t}))}\exp
\left(-c\frac{d(x,y)^2}{t}\right), \quad \forall~t>0,\, \mbox{a.e. }x,y\in
 X.
\end{equation}




Before we carry on, let us give some examples to point out that \eqref{due} is a quite common estimate:
\begin{itemize}
 \item It is well known that on a Riemannian manifold \cite[Theorem 1.1]{Gr1} or for the Laplacian 
on a subset with boundary conditions \cite{GSC}, under very weak conditions the heat kernel 
satisfies \eqref{due} and so \eqref{UE}. It is also the case for the semigroup generated by a 
self-adjoint elliptic operator of divergence form $L=-div(A \nabla)$ on the Euclidean space 
with a bounded and real valued matrix $A$ (see \cite[Theorem 4]{AT});
 \item If $(X_1, \dots, X_n)$ is a family of vector fields satisfying H\"ormander condition 
and if $H:= - \sum_{i=1}^n X_i^2$, then, in the situation of Lie groups or Riemannian 
manifolds with bounded geometry, the heat semigroup satisfies Gaussian upper-bounds 
\eqref{UE} too (see \cite[Theorem 5.14]{Robinson} and \cite[Section 3, Appendix 1]{CRT});
 \item When one consider an infinite volume Euclidean surface with conic singularities 
with $H$ equals to its Laplacian, then it is proved in \cite[Section 4]{BFHM} that 
the heat kernel satisfies Gaussian pointwise estimates \eqref{UE}.
\end{itemize}

Let us now emphasize why we put so importance on such estimates and on the heat semigroup. 
The considered operator $H$ is self-adjoint and so admits a $C^\infty$-functional calculus, 
which allows us to control $\| \phi(H) \|_{L^p \to L^p}$  for  some regular functions $\phi$. 
Such estimates can be obtained as explained in the Appendix of \cite{IP} as a consequence of 
pointwise Gaussian estimates \eqref{UE} on the heat kernel. 
Moreover, we aim to use some extrapolation techniques (to go from localized $L^2-L^2$ dispersive 
estimates to $L^p-L^{p'}$ estimates) which require local informations, as off-diagonal estimates 
of some functional operators. Such local informations could be transferred from those on the heat 
semigroup to some operators coming from a $C^\infty$-calculus (see \cite{KU} for example). However, it requires to deal with the whole class of $C^\infty$ functions (compactly supported) with suitable norms\ldots 
For an easier readability and a more intrinsic method, we prefer to work only with the semigroup 
and its time derivatives. We refer the reader to Remark \ref{rem:calculus} for the equivalence between the two points of view. 

Moreover, we still keep in mind the following very general/interesting question: what assumptions 
on the heat semigroup $(e^{-tH})_{t\geq0}$ could imply dispersive estimates and Strichartz estimates 
(possibly with a loss of derivatives) for the unitary Schr\"odinger group $(e^{-itH})_{t \in \mathbb R}$ ? 
Such a question is natural, since the application $z \mapsto e^{-zH}$ is holomorphic on 
$\{z\in {\mathbb C}, \ \Re(z)\geq 0\}$. Dispersive information on Schr\"odinger group 
should be connected to some specific properties of the heat kernel. \\

Motivated by this program, we decide to only work with the holomorphic functional 
calculus associated with the operator $H$ and more precisely, we will see that all 
of our study relies on (a sectorial) functional calculus only involving the heat 
semigroup and its time derivatives, and could also be written in terms of a $C^\infty$-calculus, see Remark \ref{rem:calculus}.

\bigskip

In the first part of this work, we investigate this question, allowing loss of derivatives, 
as in \cite{BGT}, proposing a new approach, related to the heat semigroup and the use of 
Hardy-$\BMO$ spaces associated with the semigroup. We point out that our approach gives 
an ``unified'' way to prove Strichartz estimates with loss of derivatives for the compact and noncompact framework.

\bigskip

Let us briefly explain the study of Hardy and $\BMO$ spaces associated with such a heat semigroup.
The classical spaces Hardy space $H^1$ (also called of Coifman-Weiss \cite{CW}, \cite{FS}) and $\BMO$ 
(introduced by John and Nirenberg in \cite{JN}) naturally arises (from a point of view of 
Harmonic Analysis) as a ``limit/extension" of the Lebesgue spaces scale $(L^p)_{1<p<\infty}$ 
when $p \to \infty$ (for $\BMO$) and $p \to 1$ (for $H^1$). Indeed, these two spaces have 
many properties, which are very useful and which fail for the critical spaces $L^1$ and 
$L^\infty$, as Fourier characterization, duality, boundedness of some maximal functions 
or Calder\'on-Zygmund operators, equivalence between several definitions\ldots
Even if $\BMO$ is strictly containing $L^\infty$ and $H^1$ stricly contained in $L^1$, 
these spaces still satisfies a very convenient interpolation results: indeed $H^1$ or 
$\BMO$ interpolates with Lebesgue spaces $L^p$, $1 < p < \infty$ and the intermediate 
spaces are the corresponding intermediate Lebesgue spaces. 

However, there are situations where these spaces $H^1$ and $\BMO$ are not the right 
substitutes to $L^1$ or $L^\infty$ (for example it can be shown that the Riesz 
transform may be not bounded from $H^1$ to $L^1$) and there has been recently 
numerous works whose goal is to define Hardy and BMO spaces adapted to the context 
of a semigroup (see \cite{Aus1, ACDH, Ber, BZ, BZ2, DY1, DY2, HM}). In \cite{Ber, BZ}, 
Bernicot and Zhao have described a very abstract theory for Hardy spaces (built via 
atomic decomposition) and interpolation results with Lebesgue spaces. The main idea 
is to consider the oscillation given by the semigroup instead of the classical oscillation 
involving the average operators. Then these adapted Hardy and $\BMO$ spaces have 
been extensively studied these last years (see the previous references) and it is 
known that interpolation property still holds. We refer the reader to Subsection 
\ref{sectionhardybmo} for the precise definition and refer to the previous citations for more details on this theory.

\bigskip

Still having in mind to connect the heat semigroup with the Schr\"odinger propagator, 
we aim to prove $H^1-\BMO$ dispersive estimates, using these spaces corresponding to 
the heat semigroup. Even if we loose the endpoint ($L^1-L^\infty$ estimate) by 
interpolation, we know that we could at least recover the intermediate $L^p-L^{p'}$ 
dispersive estimates. Moreover, such approach has the advantage that we do not require 
any pointwise estimates on the Schr\"odinger propagators. We first point out that such 
Hardy-$\BMO$ approach have already been used in \cite{MT,Taylor} to obtain dispersive 
estimates, but there the authors considered the classical spaces and not the ones associated with the heat semigroup.

It seems to us that the combination of dispersive estimates and $H^1-\BMO$ spaces 
associated with the heat semigroup is a new problematic. Due to the novelty of such 
approach, we first describe it in a very general setting, by introducing the following 
notion: we say that a $L^2$-bounded operator $T$ satisfies Property \eqref{disp} for some 
integer $m\geq 0$ and constant $A$ (which is intended to be $|t|^{-\frac d2}$ in the 
applications for dispersive estimates), if for every $r>0$ (or $r\lesssim \diam(X)$ if $X$ is bounded)
\begin{equation}\label{disp} \tag{$H_m(A)$}
       \|T \psi_{m}(r^2 H)\|_{L^2(B_r)\to L^2(\widetilde{B_r})} \lesssim A \mu(B_r)^{\frac 12} \mu(\widetilde{B_r})^{\frac12}
  \end{equation}
where $B_r$ and $\widetilde{B_r}$ are any two balls of radius $r$ and $\psi_{m}(x):=x^me^{- x}$.

Under
a uniform lower control of the volume: it exists $\nu>0$ such that
\begin{equation} \label{d}
     r^\nu \lesssim \mu(B(x,r)) \qquad \forall x\in X,r\lesssim \min(1,\diam(X)),
 \end{equation}
we then prove the following:

\begin{thm} \label{thm1} Assume \eqref{dd}, \eqref{d} with \eqref{due}. 
Consider a self-adjoint and $L^2$-bounded operator $T$ (with $\|T\|_{L^2 \to L^2} \lesssim 1$), 
which commutes with $H$ and satisfies Property \eqref{disp} for  some $m \geq \frac d2 $. 
Then $T$ is bounded from $H^1$ to $BMO$ and from $L^p$ to $L^{p'}$ for $p\in(1,2)$ with 
$$ \|T\|_{H^1 \to BMO} \lesssim A \qquad \textrm{and} \qquad \|T\|_{L^p \to L^{p'}} \lesssim A^{\frac{1}{p}-\frac{1}{p'}}$$
if the ambient space $X$ is unbounded and
$$ \|T\|_{H^1 \to BMO} \lesssim \max(A,1) \qquad \textrm{and} \qquad \|T\|_{L^p \to L^{p'}} \lesssim \max(A^{\frac{1}{p}-\frac{1}{p'}},B)$$ 
if the ambient space $X$ is bounded, and where, for the last inequality, we assumed that $\|T\|_{L^p \to L^2}\lesssim B$.
\end{thm}

As a consequence, this allows us to reduce $L^p-L^{p'}$ dispersive estimates to microlocalized $L^2-L^2$ estimates 
(localized in the frequency through the operators $\psi(r^2H)$ and in the physical space through 
the balls $B_r$ and $\widetilde{B_r}$, respecting the Heisenberg uncertainty principle).

Note that Property \eqref{disp} is weaker (and also necessary to have) than a $L^1-L^\infty$ estimate: 
indeed if $T$ is supposed to map $L^1$ to $L^\infty$ with a bound lower than $A$ then Cauchy-Schwarz inequality leads us to Property \eqref{disp}:
\begin{align*}
 \|T \psi(r^2H)f\|_{L^2(\widetilde{B_r})} &\leq A  \| \psi(r^2H)f\|_{L^1}  \mu(\widetilde{B_r})^{1/2} \\  
& \lesssim A \mu(\widetilde{B_r})^{\frac 12} \|f\|_{L^1(B_r)} \\
 &\leq A \mu(\widetilde{B_r})^{\frac 12} \mu(B_r)^{\frac 12} \|f\|_{L^2(B_r)}.
\end{align*}
For the $L^1-L^1$ continuity of $\psi(r^2H)$ see Corollary \ref{coro_continuite_semigroupe}.


\medskip

Our goal is to obtain the dispersive estimate $$\|T_t(H)\|_{H^1 \to \BMO} \lesssim |t|^{-\frac
d2},$$ where $T_t(H)=e^{itH} \psi_{m}(h^2H)$, $h>0$. 
The case of the full range $|t|\leq 1$ (which is independent of $h$) is the most difficult. 
However the case $|t|\leq h^2$ is straightforward: indeed for $m=0$ ($m\neq0$ then deduces easily) 
we have $T_t(H)=e^{itH}e^{-h^2H}=e^{-zH}$ with $z=h^2-it$. The key observation is that $|z|=\sqrt{h^4+t^2}\lesssim h^2 = \text{Re}(z)$. 
Thus (the complex time $z$ lives into a sector far away from the axis of imaginary complex numbers) 
by analyticity, Property \eqref{UE} can be extended to complex time semigroup and so
$$\|T_t(H)\|_{L^1 \to L^{\infty}} \leq \frac{1}{\text{Re}(z)^{\frac d2}}\lesssim
\frac{1}{|z|^{\frac d2}}\leq \frac{1}{|t|^{\frac d2}}.$$
So the full $L^1- L^\infty$ and so Property \eqref{disp} (which is weaker, as we have just seen) are obviously satisfied in the range $|t|\lesssim h^2$. 

The intermediate case $h^2 \lesssim |t|\leq h$ is treated in the particular case of compact
Riemannian manifolds in \cite{BGT} together with the implicated Strichartz estimates with a nontrivial loss of derivatives. 
We will focus on this interesting situation and we describe in this very general setting how dispersive estimates imply these Strichartz estimates (see Theorem \ref{thm_strichartz}).

\bigskip

We intend to emphasize the link between the heat semigroup and the wave operator. 
In the second part of the paper (from Section \ref{section_wave_propagation}) we aim 
to study what dispersive properties on the wave equation would be sufficient to ensure 
our hypothesis \eqref{disp} for $T=e^{itH}$, in order to regain the dispersive estimates 
and then Strichartz estimates for the Schr\"odinger group. Mainly, we are interested in 
the wave propagator $\cos(t\sqrt{H})$ which is defined as follows: for any $f\in L^2$, 
$u(t):=t\mapsto \cos(t\sqrt{H})f$ is the unique solution of the wave equation:
$$\begin{cases}
&\partial_t^2u +Hu = 0 \\
&u_{|t=0}=f\\
&\partial_t u_{|t=0}=0. \end{cases} $$
One can find the explicit solutions of this problem in \cite{Fol} for the Euclidean 
case and in \cite{Berard} for the Riemannian manifold case through precise formulae 
for the kernel of the wave propagator. 
The remarkable property of this operator comes from its finite speed propagation. 
We know that Davies-Gaffney estimates \eqref{davies_gaffney_estimates} imply 
(and indeed are equivalent (\cite[Theorem 3.4]{CS}) to) the finite speed propagation 
property at a speed equals to $1$: namely, for every disjoint open subsets $U_1,U_2 \subset X$, 
every function $f_i \in L^2(U_i)$, $i=1,2$, then
\begin{equation}\label{speed_propagation}
 \langle \cos(t\sqrt H)f_1,f_2 \rangle =0
\end{equation}
for all $0<t<d(U_1,U_2)$. If $\cos(t\sqrt H)$ is an integral operator with kernel $K_t$, 
then \eqref{speed_propagation} simply means that $K_t$ is supported in the ``light cone'' 
$\mathcal D_t:=\{(x,y) \in X^2,\ d(x,y)\leq t\}$.

To apply Theorem \ref{thm1} in order to get dispersive estimates, we first have to prove 
that Schr\"odinger propagators satisfy Property \eqref{disp} for some suitable constant $A$. 
The following formula (see Section \ref{section_wave_propagation}):  for all $z \in \mathbb C$ with $\Re(z)>0$:
$$e^{-zH}=\frac{1}{\sqrt{\pi}} \int_0^{+\infty} \cos(s \sqrt H) e^{-\frac{s^2}{4z}} \frac{ds}{\sqrt z},$$
alllows us to describe the link between Schr\"odinger propagators and wave propagators.

We will need extra assumptions (deeper than just the finite speed propagation property), 
in order to be able to check Property \eqref{disp}. More precisely, we need the following 
short time $L^2-L^2$ dispersive estimates:
\begin{hyp}\label{hyp_cossH_intro} There exist $\kappa\in(0,\infty]$ and an integer $m_0$ 
such that for every $s\in(0,\kappa)$ we have: for every $r>0$, every balls $B_r$, $\widetilde{B_r}$ 
of radius $r$ then
$$\|\cos(s\sqrt H) \psi_{m_0}(r^2H)\|_{L^2(B_r) \to L^2(\widetilde{B_r})} \lesssim \left(\frac{r}{s+r}\right)^{\frac{d-1}{2}} \left( 1+\frac{|L-s|}{r} \right)^{-\frac{d+1}{2}}$$
 where $L=d(B_r,\widetilde{B_r})$.
 \end{hyp}
To obtain our second result we will need more regularity on the measure than \eqref{d}. 
For the next theorem assume that $\mu$  is Ahlfors regular: there exist two 
absolute positive constants $c$ and $C$ such that for all $x \in X$ and $r > 0$:
\begin{equation} \label{ahbis}
 c r^{d} \leq \mu(B(x, r)) \leq C r^{d}.
\end{equation}

Then our second main theorem is the following:
\begin{thm}\label{demo_hyp_Hmn2_intro} Suppose \eqref{ahbis} with $d>1$, \eqref{due} and
Assumption \ref{hyp_cossH_intro} with $\kappa\in(0,\infty]$. Then for every integer $m\geq \max(\frac d2, m_0+\left\lceil\frac{d-1}{2}\right\rceil)$ we have
\begin{itemize}
\item if $\kappa=\infty$: the propagator $e^{itH}$ satisfies Property $(H_m(|t|^{-\frac d2}))$
for every $t\in {\mathbb R}$ and so we have Strichartz estimates without loss of derivatives;
\item if $\kappa<\infty$: for every $\varepsilon>0$, every $h>0$ with
$|t|<h^{1+\varepsilon}$ and integer $m'\geq 0$ the propagator $e^{itH} \psi_{m'}(h^2H)$ satisfies Property $(H_m(|t|^{-\frac d2}))$
and so we have Strichartz estimates with loss of $\frac{1+\varepsilon}{p}$ derivatives.
\end{itemize}
\end{thm}

It is worth noting that, in the proof, the same approach raises the two cases: 
\begin{itemize}
 \item $\kappa<+\infty$ which leads to Strichartz estimates with loss of derivatives,
 \item $\kappa=+\infty$ which leads to estimates without loss of derivatives.
\end{itemize}
Moreover, Assumption \ref{hyp_cossH_intro} holds in the context of smooth compact Riemannian 
manifold with $\kappa$ given by the injectivity radius and in the Euclidean situation with 
$\kappa = \infty$ (also with smooth perturbation and there $\kappa<\infty$).

As an example, note that in the case of the Euclidean space with an operator of the form 
$H=-\textrm{div} A \nabla$, where $A$ is a matrix with variable and $C^{1,1}$ coefficients, 
then Smith has built a short time parametrix \cite{Smith} of the corresponding wave equation 
(see also the work of Blair \cite{Blair}), which yields in particular our Assumption \ref{hyp_cossH_intro} 
for some $\kappa<\infty$. As a consequence, we deduce that the solutions of Schr\"odinger equation satisfies Strichartz 
estimates with loss of $\frac{1+\varepsilon}{p}$ derivatives for every $\varepsilon>0$. 

By this way, we have an unified approach to deal with compact or noncompact situations and we recover 
(up to a loss $\varepsilon$ as small as we want) the Strichartz estimates with loss of derivatives for 
a compact smooth manifold due to \cite{BGT,ST} and full Strichartz estimates for the Euclidean situation.

The plan of this article is as follow: In Section \ref{section_definition} we first set the 
notations and definitions used throughout the paper. Then we describe the assumptions required 
on the heat semigroup $e^{-tH}$ together with some basic properties about Hardy-$\BMO$ spaces 
and functional calculus associated to $H$. 
Theorem \ref{thm1} is proved in Section \ref{section_dispersion} and we apply it in Section 
\ref{section_strichartz} to prove Strichartz estimates (with a possible loss of derivatives). 
Section \ref{section_wave_propagation} shows Theorem \ref{demo_hyp_Hmn2_intro}, and Section 
\ref{section_euclidean_case} how the hypothesis \eqref{disp} can be derived from the small 
time parametrix of the associated wave operator.

\section{Definitions and preliminaries}\label{section_definition}

\subsection{Notations}

For $B(x,r)$ a ball ($x\in X$ and $r>0$) and any parameter $\lambda>0$, we denote
$\lambda B(x,r) := B(x,\lambda r)$ the dilated and concentric ball. 
As a consequence of the doubling property, a ball $B(x,\lambda r)$ can be covered by
$C\lambda^d$ balls of radius $r$, uniformly in $x \in X$, $r>0$ and $\lambda>1$ ($C$
is a constant only dependent on the ambient space).
Moreover, the volume of the balls satisfies the following behaviour:
\begin{equation}\label{ball_behaviour}
\mu(B(y,r))\lesssim \left (1+\frac{d(x,y)}{r} \right)^d \mu(B(x,r))
\end{equation}
uniformly for all $x,y \in X$ and $r>0$. \\
For a ball $Q$, and an integer $i\geq1$, we denote $C_i(Q)$ the $i$th
dyadic corona around $Q$: 
$$C_i(Q):=2^iQ \backslash 2^{i-1}Q.$$ We also set $C_0(Q)=Q$.\\
If no confusion arises, we will note $L^p$ instead of $L^p(X, \mu)$ for $p \in [1 
\virg \infty]$. \\
We will use $u\lesssim v$ to say that there exists a constant $C$ (independent of
the important parameters) such that $u\leq Cv$ and $u\simeq v$ to say that
$u\lesssim v$ and $v\lesssim u$. \\
If $\Omega$ is a set, $\mathbf1_\Omega$ is the characteristic function of $\Omega$,
defined by
 $$\mathbf1_\Omega(x)=
\begin{cases}
1\textrm{ if }x \in\Omega\\
0\textrm{ if }x\notin\Omega. \\
\end{cases}$$
The Hardy-Littlewood maximal operator is denoted by $\calM$ and is given for every
$x\in X$ and function $f\in L^1_{loc}$ by:
$$ \calM(f)(x):= \sup_{\overset{B \textrm{ ball}}{x\in B}} \ \left(\frac{1}{\mu(B)} \int_B |f| d\mu\right).$$
Since the space is of homogeneous type, it is well-known that this maximal operator
is bounded in any $L^p$ spaces, for $p\in(1,\infty]$.

\subsection{The heat semigroup and associated functional calculus}

We recall the definition of a $L^2$-holomorphic semigroup:
\begin{definition}\label{def_semi-groupe}
A familly of operators $\left(S(z)\right)_{\Re(z)\geq 0}$ on $\mathcal L(L^2)$ is
said to be a holomorphic semigroup on $L^2$ if (with $\Gamma:=\{z\in{\mathbb C},\
\Re(z)\geq 0\}$): 
\begin{enumerate}
 \item $S(0)= id$;
 \item $\forall z_1, z_2 \in \Gamma \virg S(z_1+z_2)=S(z_1) \circ S(z_2)$;
 \item $\forall f \in L^2 \virg \lim \limits_{\genfrac{}{}{0pt}{}{z \to
0}{z\in\Gamma}} \|S(z)f-f\|_{L^2}=0$;
 \item $\forall f,g \in L^2$, the map $z \mapsto \langle S(z)f,g\rangle$ is
holomorphic on $\Gamma$.
\end{enumerate}
\end{definition}

We recall the bounded functional calculus theorem from \cite{RS}: 
\begin{thm}\label{thm_calcul_fonctionnel}
Since $H$ is a nonnegative self-adjoint operator, it admits a $L^\infty$-functional
calculus: if $f \in L^{\infty}(\mathbb R_+)$, then we may consider the
operator $f(H)$ as a $L^2$-bounded operator and
$$\|f(H)\|_{L^2 \to L^2}\leq \|f\|_{L^{\infty}}.$$
\end{thm}

For any integer $m \geq 1$ and real $n>0$, we set $\psi_{m,n}(x)=x^me^{-nx}$ and
$\psi_m:=\psi_{m,1}$. These smooth functions $\psi_{m,n} \in \mathcal
C^{\infty}(\mathbb R^+)$, vanish at $0$ and at infinity; moreover
$\|\psi_{m,n}\|_{L^{\infty}({\mathbb R}_+)}\lesssim 1$.
The previous theorem allows us to define the operators $\psi_{m,n}(tH)$ for any
$t\geq0$ and $m \in \mathbb N$, $n>0$. 

From the Gaussian estimates of the heat kernel \eqref{UE} and the analyticity of the
semigroup (see \cite{CCO}) it comes that for every integer $m\in{\mathbb N}$ and
$n>0$ the operator $\psi_{m,n}(tH)$ has a kernel $p_{m,n,t}$ also satisfying upper
Gaussian estimates:
\begin{equation}
|p_{m,n,t}(x,y)|\lesssim
\frac{1}{\mu(B(x,\sqrt{t}))}\exp
\left(-c \frac{d(x,y)^2}{t}\right), \quad \forall~t>0,\, \mbox{a.e. }x,y\in
 X.\label{UEmn}
\end{equation}

We now give some basic results about the semigroup thanks to our assumptions.

\begin{prop}\label{prop_continuite_semigroupe} Under \eqref{dd} and \eqref{UE}, the
heat semigroup is pointwisely bounded by the Hardy-Littlewood maximal operator and
is uniformly bounded in every $L^p$-spaces for $p\in[1,\infty]$: for every locally
integrable function $f$ and every $x_0 \in X$, we have
$$ \sup_{t>0} \left\|e^{-tH} f \right\|_{L^\infty(B(x_0,\sqrt{t}))} \lesssim
\calM(f)(x_0)  \quad \textrm{and} \quad \sup_{t>0} \| e^{-tH} f \|_{L^p} \lesssim
\|f\|_{L^p}.$$
\end{prop}

\begin{proof}
The pointwise boundedness by the maximal function is an easy consequence of
\eqref{UE} with the doubling property \eqref{dd}. As a consequence, the
$L^p$-boundedness of the maximal operator yields the uniform $L^p$-boundedness of
the heat semigroup, for $p>1$. Let us now check the $L^1$-boundedness.
By \eqref{UE}, we have: 
\begin{align*}
 \int_{x \in X} |e^{-tH}f(x)|d\mu(x) &\lesssim \int_{x \in X} \int_{y \in X}
\frac{1}{\mu(B(x,\sqrt t))}e^{-c \frac{d(x,y)^2}{t}}|f(y)|d\mu(y) d\mu(x) \\
 &\lesssim \int_{y \in X} |f(y)| \frac{1}{\mu(B(y,\sqrt t))} \int_{x\in X} \left(
1+\frac{d(x,y)}{\sqrt t} \right)^d e^{-c \frac{d(x,y)^2}{t}}d\mu(x) d\mu(y).
\end{align*}
A decomposition in coronas around $B(y, \sqrt t)$ allows us to control the integral
over $x$: 
\begin{align*}
 &\int_{B(y,\sqrt t)} \left( 1+\frac{d(x,y)}{\sqrt t} \right)^d e^{-c
\frac{d(x,y)^2}{t}}d\mu(x)+\sum_{j\geq1}\int_{C_j(B(y,\sqrt t))} \left(
1+\frac{d(x,y)}{\sqrt t} \right)^d e^{-\frac{d(x,y)^2}{t}}d\mu(x)\\
 \leq& 2^d \mu(B(y,\sqrt t))+\sum_{j\geq1} \left( 1+2^j \right)^d e^{-c 2^{2j}}
\mu(B(y, 2^j \sqrt t))\\
 \lesssim& \left( 2^d + \sum_{j\geq1} (1+2^j)^d 2^{jd} e^{-c 2^{2j}} \right)
\mu(B(y,\sqrt t)) \lesssim \mu(B(y,\sqrt t)),
\end{align*}
where the last line results from the doubling property of $\mu$. 
Hence,  uniformly in $t>0$ $$\|e^{-tH}\|_{L^1 \to L^1} \lesssim 1.$$
\end{proof}

\begin{corollaire} \label{coro_continuite_semigroupe} For $m\in{\mathbb N}$ and
$n>0$, since $\psi_{m,n}(tH)$ satisfies \eqref{UE}, we deduce that
the operators $\psi_{m,n}(tH)$ also satisfy the same estimates.
\end{corollaire}

Let us now give some basic properties about the functions $\psi_{m,n}$: 
\begin{prop}\label{prop_basique}
 (a) for all $k \in \mathbb N^*$, $m\in{\mathbb N}$ and $n>0$ then 
$\psi_{km,kn}=\left(\psi_{m,n}\right)^k$;\\
 (b) for all $m,m'\in{\mathbb N}$ and $n,n',u,v>0$  then :
$$\psi_{m,n}(u\cdot)\psi_{m',n'}(v\cdot)=\frac{u^mv^{m'}}{(nu+n'v)^{m+m'}}\psi_{m+m',1}((nu+n'v)\cdot);$$
 (c) for every $r>0$ and every $f\in L^2$ then: $$(1-e^{-r^2H})f=\int_0^{r^2}
He^{-sH}fds=\int_0^{r^2} \psi_{1,1}(sH)f \frac{ds}{s};$$
 (d) for $m\in{\mathbb N^*}$, $n>0$ and $f\in L^2$, then $\left( \int_0^{+\infty}
\|\psi_{m,n}(vH)f\|_{L^2}^2 \frac{dv}{v} \right)^{\frac 12} + \|P_{N(H)}f\|_{L^2} \lesssim \|f\|_{L^2}$; 
where $P_{N(H)}$ is the projector on the kernel of $H$: $N(H):=\{f\in L^2 \cap {\mathcal D}(H),\ Hf=0\}$.\\
 (e) for $m\in{\mathbb N^*}$, $n>0$ up to a constant $c_{m,n}$, we have the
decomposition: $$Id=c_{m,n} \int_0^{+\infty} \psi_{m,n}(sH) \frac{ds}{s}+P_{N(H)}.$$
\end{prop}
\begin{proof}
 (a), (b), (c) and (e) are straightforward. (d) is classical and a direct
application of (e) with the almost-orthogonality of $\psi_{m,n}(vH)$ operators: for
every $u,v>0$
 $$ \|\psi_{m,n}(uH) \psi_{m,n}(vH)\|_{L^2\to L^2} \lesssim
\min\left(\frac{u}{v},\frac{v}{u}\right)^{m};$$
for which we refer to \cite{BBR} e.g.
\end{proof}

It is crucial to keep in mind that by the holomorphic functional calculus, item (e)
gives a decomposition of the identity
$$ Id = c_{m,n} \int_0^\infty \psi_{m,n}(sH) \frac{ds}{s}+P_{N(H)},$$
which has to be seen/thought as a smooth version of the spectral decomposition.
Indeed the operator $\psi_{m,n}(sH)$ plays the role of a regularized version of the
projector $\mathbf1_{[s^{-1}, 2s^{-1}]}(H)$.

\begin{rem} \label{rem:calculus}
We would like to emphasize that the use of $\psi_{m,n}$ functions is exactly equivalent to the use of smooth compactly supported cut-off functions. 
Indeed, it is easy by a smooth partition of the unity to build $\psi_{m,n}$ by an absolutely convergent serie 
of smooth and compactly supported cut-off functions. From functional calculus, we also know how we can build 
a smooth and compactly supported function by the resolvent of $H$ (using the semigroup) and so the $\psi_{m,n}$ functions 
(see \cite[Appendix]{IP} or \cite{KU} e.g.).

We have chosen to work with $\psi_{m,n}$ functions to enlighten the connexion between dispersive estimates 
and heat semigroup and also to get around the different norms that we have to consider on the $C^\infty$ space.
\end{rem}

\subsection{Quadratic functionals associated to the heat semigroup and Sobolev spaces}

Let us define some tools for the next theorem: 
$$\varphi (\lambda)=\int_{\lambda}^{+\infty} \psi_{m,n}(u)\frac{du}u,$$
$$\tilde
\varphi(\lambda)=\int_0^{\lambda}\psi_{m,n}(v)\frac{dv}v=\int_0^1\psi_{m,n}(\lambda
u)\frac{du}u.$$

\begin{rem}
 Notice that $\varphi$ is, by integration by parts, a finite linear combination of
functions $\psi_{k,\ell}$ for $k\in\{0,..,m\}$ and $\ell>0$. Moreover for every $\lambda
\in \mathbb R$,  
 $$\tilde
\varphi(\lambda)+\varphi(\lambda)=\int_0^{+\infty}u^{m-1}e^{-nu}du=\frac{\Gamma(m)}{n^m}:=c_{m,n}.$$
\end{rem}
The following theorem will be useful to estimate the $L^p$-norm through the heat
semigroup: 

\begin{thm}\label{thm_LP} Assume \eqref{dd} and \eqref{due}. For every integer
$m\geq 2$, real number $n>0$ and all $p\in(1,\infty)$, we have
\begin{align*}
 \|f\|_{L^p} & \simeq  \|\varphi(H)f\|_{L^p}+\left\| \left(\int_0^1
|\psi_{m,n}(uH)f|^2\frac{du}u\right)^{\frac 12} \right\|_{L^p}. 
\end{align*}
So if $q\geq 2$
\begin{align*}
 \|f\|_{L^q} & \lesssim  \|\varphi(H)f\|_{L^q}+ \left(\int_0^1
\|\psi_{m,n}(uH)f\|_{L^q}^2\frac{du}u\right)^{\frac 12}. 
\end{align*}
\end{thm}

Such a result can be seen as a semigroup version of the Littlewood-Paley
characterization of Lebesgue spaces.

\begin{proof}  We give the sketch of the proof (and refer to \cite[Chapter 6,
Theorem 6.1]{Aus1} and \cite[Proposition 2.12]{BBR} for more details where it is
proved that such inequalities hold for every exponent $p$ belonging to the range
dicted by the heat semigroup $e^{-tH}$ ; here \textcolor{red}{$(1,\infty)$}). We aim to study the
boundedness of the quadratic functional
$$T \colon f \mapsto \left( \int_0^1|\psi_{m,n}(s^2H)f|^2\frac{ds}s\right)^{\frac12}.$$
Indeed $T$ is a horizontal square function (or Littlewood-Paley-Stein
$g$-function), and its $L^p$-boundedness is well-known by functional calculus
theory (see \cite{Ste}, \cite{meda} and references therein) when the semigroup is
submarkovian and conservative.

We aim here to quickly explain another approach (more analytic) of its boundedness,
which does not require submarkovian property and conservativeness but relies on Gaussian estimates
\eqref{UE}. We are looking to apply extrapolation result \cite[Theorem 1.1]{Aus1}
or \cite{BZ} to $T$ with $p_0=1$. To keep the notation of \cite{Aus1}  we recall
that 
$$A_r:=Id-(Id-e^{-r^2H})^M \quad \textrm{and} \quad B_r:=Id-A_r=(Id-e^{-r^2H})^M,$$
with $M$ a large enough integer.
First by $L^2$ holomorphic functional calculus, it is known that $T$ is bounded on
$L^2$ (see \cite{BBR} e.g.). We now have to check the two main hypothesis of
\cite[Theorem 1.1]{Aus1}.

By expanding, $A_r$ behaves like $e^{-r^2H}$, in the sense that it admits a kernel
satisfying the Gaussian upper estimates at the scale $r$. Note $B$ a ball of radius
$r$, and $f$ supported in $B$. For all $j\geq1$ and for all $x \in C_j (B)$, we have
\begin{align*}
|A_rf(x)|&\lesssim \int_B \frac{1}{\mu(B(x,r))}e^{-c\frac{d(x,y)^2}{r^2}}|f(y)|d\mu(y) \\
 &\lesssim \frac{1}{\mu(B(x,r))} \int_B e^{-c2^{2j}}|f|d\mu.
\end{align*}
If $z$ denotes the center of $B$, then by the doubling property
\eqref{ball_behaviour} of the measure it comes
$$\mu(B) \lesssim \left( 1+ \frac{d(z,x)}{r}\right)^d \mu(B(x,r)),$$ so that
$$\mu(B(x,r))^{-1} \lesssim \left( 1+\frac{d(z,x)}{r} \right)^d \mu(B)^{-1} \lesssim
2^{jd}\mu(B)^{-1},$$
where we used that $z\in B$ and $x \in C_j(B)$ so $d(z,x) \lesssim 2^jr$.
Hence, $$\left( \frac1{\mu(2^{j+1}B)}\int_{C_j(B)}|A_{r}f|^2d\mu
\right)^{\frac12} \lesssim g(j) \frac{1}{\mu(B)}\int_B|f|d\mu,$$
with $g(j)\lesssim e^{-c 4^j} 2^{jd}$ satisfying
 $$\sum_j g(j)2^{dj}<+\infty.$$
That is the first assumption required in \cite[Theorem 1.1]{Aus1}.

Then the second (and last) assumption of \cite[Theorem 1.1]{Aus1} has been weakened
in \cite{BZ} and we only have to check that for all $j \geq 2$: $$\left(
\frac1{\mu(2^{j+1}B)}\int_{C_j(B)}|T(B_r f)|^2d\mu \right)^{\frac12}\leq g(j) \left(
\frac{1}{\mu(B)}\int_B|f|^2d\mu\right)^{1/2}.$$
We refer the reader to \cite[Step 3, item 1, Theorem 6.1]{Aus1} and also \cite{CDL}
and \cite{ACDH}, where such inequalities are proved, and the arguments only rely on
the Davies-Gaffney estimates \eqref{davies_gaffney_estimates} for $\psi_{m,n}(tH)$.

By this way, we may apply \cite[Theorem 1.1]{Aus1} and deduce that the square
function $T$ is bounded on $L^p$ for every $p\in(1,2]$.
For $p>2$, we have to apply \cite[Theorem 1.2]{Aus1} and this is also detailed in
\cite[Step 2, Theorem 6.1]{Aus1}.
Thus, if $p\in(1,\infty)$ then 
\begin{equation} \label{eq:squ} \left\| \left( \int_0^1
|\psi_{m,n}(s^2H)f|^2\frac{ds}s\right)^{\frac12}\right \|_{L^p} \lesssim
\|f\|_{L^p}.\end{equation}

It remains to check the reverse inequalities. We proceed by duality to finish the
proof. Since $\varphi(x) + \int_0^1 \psi_{m,n}(tx)\frac{dt}{t} = c_{m,n}$ is a constant
independent of $x$, then:
\begin{align*}
c_{m,n} \langle f,g \rangle &=  \langle
f,\varphi(H)g+\int_0^1\psi_{m,n}(tH)g\frac{dt}t \rangle \\
 &= \langle \varphi(H)f,g \rangle +\int_0^1 \langle \psi_{\frac m2, \frac
n2}(tH)f,\psi_{\frac m2, \frac n2}(tH)g \rangle \frac{dt}t. 
\end{align*}
We should decompose $m=m_1+m_2$ with 2 integers $m_1$, $m_2$ comparable to $\frac
m2$. For simplicity we take $m_1=m_2=\frac m2$, assume they are integers, and let to
the reader the minor modifications. The Cauchy-Schwarz inequality for the scalar
product $(u,v)=\int_0^1 u(t)v(t)\frac{dt}t$ gives then: 
\begin{align*}
|\langle f,g \rangle | & \lesssim |\langle \varphi(H)f,g\rangle |+ \int \left(
\int_0^1|\psi_{\frac m2,\frac n2}(tH)f|^2\frac{dt}t\right)^{\frac12} \left(
\int_0^1|\psi_{\frac m2,\frac n2}(tH)g|^2\frac{dt}t\right)^{\frac12} d\mu \\
 &\leq \| \varphi(H)f\|_{L^{p}} \|g\|_{L^{p'}}+ \left \| \left( \int_0^1|\psi_{\frac
m2,\frac n2}(tH)f|^2\frac{dt}t \right)^{\frac12}\right\|_{L^{p}} \left \| \left(
\int_0^1|\psi_{\frac m2,\frac n2}(tH)g|^2\frac{dt}t
\right)^{\frac12}\right\|_{L^{p'}} \\
 &\lesssim \| \varphi(H)f\|_{L^{p}} \|g\|_{L^{p'}}+ \left \| \left(
\int_0^1|\psi_{\frac m2,\frac n2}(tH)f|^2\frac{dt}t
\right)^{\frac12}\right\|_{L^{p}} \|g\|_{L^{p'}},
\end{align*}
where $\frac 1p+\frac 1{p'}=1$ and we used \eqref{eq:squ} for $p'$. \\
Thus, by duality
$$\|f\|_{L^{p}}=\underset{\|g\|_{L^{p'}}\leq 1}{\sup}|<f,g>| \lesssim \|
\varphi(H)f\|_{L^{p}} +\left \| \left( \int_0^1|\psi_{\frac m2,\frac
n2}(tH)f|^2\frac{dt}t \right)^{\frac12}\right\|_{L^{p}}.$$
That concludes the proof of the characterization of the Lebesgue norms, via these
square functionals.

Then in particular for $q\in[2,+\infty)$, Minkowski generalized inequality finally
gives $$\|f\|_{L^q(M)}\lesssim \|
\varphi(H)f\|_{L^q}+\left(\int_0^1\|\psi_{m,n}(uH)f\|_{L^q(M)}^2\frac{du}u\right)^{\frac
12}.$$
\end{proof}

We will also work with the nonhomogeneous Sobolev spaces associated to $H$, defined
in terms of Bessel type: for $s\geq 0$ and $p\in(1,\infty)$, $W_H^{s,p}$ which will
be noted $W^{s,p}$ is the Sobolev space of order $s$ associated/equipped with the
norm 
$$ \|f\|_{W^{s,p}} := \|(1+H)^{\frac s2}f\|_{L^p} \simeq \|f\|_{L^p} + \| H^{\frac
s2}f \|_{L^p}.$$
Following the previous result, it comes
\begin{align*}
\|f\|_{W^{s,p}} &\simeq \|\varphi(H)f\|_{L^p}+ \left\| \left( \sum_{k=1}^{+\infty}
2^{2ks}|\psi_{m,n}(2^{-2k}H)f|^2\right)^{\frac 12}\right\|_{L^p}\\
                &\simeq \|\varphi(H)f\|_{L^p}+ \left\| \left( \int_0^1 u^{-2s}
|\psi_{m,n}(u^2H)f|^2\frac{du}u\right)^{\frac 12} \right\|_{L^p}
\end{align*}

We refer the reader to \cite{BBR} for more details about such Sobolev spaces.
We can move from the discrete to the continuous case of those partitions of
``Littlewood-Paley'' writing:  
\begin{align*}
\sum_{k=1}^{+\infty} \int_1^2
\psi_{m,n}(2^{-2k}u^2\lambda)\frac{du}u&=\sum_{k=1}^{+\infty}
\int_{2^{-k}}^{2^{-(k-1)}}\psi_{m,n}(\lambda v^2)\frac{dv}v
=\int_0^1\psi_{m,n}(\lambda v^2)\frac{dv}v \\
        &=\int_0^{\lambda}\psi_{m,n}(u)\frac{du}{2u} =\int_0^1 \psi(\lambda u)\frac{du}{2u}.
\end{align*}

\begin{rem}
The left hand integral works over $u\in[1,2]$, so we can pass from information on
the discrete case to the same information on the continuous case.
\end{rem}

 \subsection{Hardy and BMO spaces}\label{sectionhardybmo}

We define now atomic Hardy spaces adapted to our situation (dicted by a semigroup on
a doubling space) using the construction introduced in \cite{BZ}. Let $\mathcal Q$
be the family of all balls of $X$: 
$$\mathcal Q := \{ B(x,r) \virg x\in X \virg r>0 \}.$$
We define $\left( B_Q \right)_{Q \in \mathcal Q}$ the family of operators: $$\forall
Q \in \mathcal Q \virg B_Q:=(1-e^{-r^2H})^M,$$ where $r$ is the radius of the ball
$Q$ and $M$ is an integer (large enough: $M\geq \min(\frac34+\frac{3d}{8},3)$ is
sufficient). 
Those operators are bounded on $L^{2}$ uniformly in $r$. Indeed, by expanding, $B_Q$ 
is a finite linear combination of operators $e^{-kr^2H}$ with $k \in \{ 0,
\ldots, M \}$ and Theorem \ref{thm_calcul_fonctionnel} gives $$\|e^{-kr^2H}\|_{L^2
\to L^2} \leq \|x \mapsto e^{-kr^2x}\|_{L^{\infty}(\mathbb R_+)} \leq 1,$$ because
$H$ is nonnegative. 
\begin{rem} \label{rem:M}
$M\geq \frac 34 + \frac{3d}{8}$ ensures that $\frac{4M}{3} - \frac d2 \geq 1$ so
there exists an integer $m \in [\frac d2,\frac{4M}{3}]$.
\end{rem}

\begin{definition}\label{def_atome}
 A function $a \in L^1_{loc}$ is an atom associated with the ball $Q$ if there exists
a function $f_Q$ whose support is included in $Q$ such that $a=B_Q(f_Q)$, with 
 $$\|f_Q\|_{L^{2}(Q)} \leq ( \mu(Q) )^{-\frac{1}{2}}.$$
\end{definition}
That last condition allows us to normalize $f_Q$ in $L^1$. Indeed by the
Cauchy-Schwarz inequality $$\|f_Q\|_{L^1}\leq \|f_Q\|_{L^2(Q)} \mu(Q)^{\frac 12}
\leq 1.$$
Moreover, $B_Q$ is bounded on $L^{1}$ so every atom is in $L^{1}$ and they are also
normalized in $L^1$
\begin{equation}\label{thm_atome_L1} \sup_{a} \|a\|_{L^1} \lesssim 1, 
\end{equation}
where we take the supremum over all the atoms. Indeed, consider an atom 
$a=B_Q(f)=(1-e^{-r^2H})^Mf$ with suitable function $f$ supported on a ball $Q$.
By the binomial theorem, $B_Q$ behaves like $e^{-kr^2H}$. So Proposition
\ref{prop_continuite_semigroupe} gives 
$$\|a\|_{L^1(X)}=\|B_Q(f)\|_{L^1(X)} \leq \sum_{k=1}^M \binom Mk
\|e^{-kr^2H}f\|_{L^1} \lesssim \|f\|_{L^1} \lesssim 1.$$

We may now define the Hardy space by atomic decomposition:

\begin{definition}\label{def_hardy}
 A measurable function $h$ belongs to the atomic Hardy space $H^1_{ato}$, which will
be denoted $H^1$, if there exists a decomposition $$h=\sum_{i \in \N} \lambda_i a_i
\quad \mu- \textrm{a.e.}$$
 where $a_i$ are atoms and $\lambda_i$ real numbers satisfying: $$\sum_{i \in \N}
|\lambda_i| < + \infty.$$
 We equip the space $H^1$ with the norm: $$\|h\|_{H^1}:= \inf \limits_{h=\sum_i
\lambda_i a_i}\sum_{i\in \N}|\lambda_i|,$$
where we take the infimum over all the atomic decompositions.
\end{definition}

For a more general definition and some properties about atomic spaces we refer to
\cite{Ber,BZ}, and the references therein.  From \eqref{thm_atome_L1}, we deduce:
\begin{corollaire} \label{injH1L1} The Hardy space is continuously embedded into $L^1$:
$$ \|f\|_{L^1} \lesssim  \|f\|_{H^1}.$$
From \cite[Corollary 7.2]{BZ}, the Hardy space $H^1$ is also a Banach space.
\end{corollaire}

We refer the reader to \cite[Section 8]{BZ}, for details about the problem of
identifying the dual space $(H^1)^*$ with a $\BMO$ space. For a $L^\infty$-function,
we may define the $\BMO$ norm
$$ \|f\|_{\BMO} := \sup_{Q} \left(\aver{Q} |B_Q(f)|^2 \, d\mu \right)^{1/2},$$
where the supremum is taken over all the balls. If $f\in L^\infty$ then $B_Q(f)$ is
also uniformly bounded (with respect to the ball $Q$), since the heat semigroup is
uniformly bounded in $L^\infty$ (see Proposition \ref{prop_continuite_semigroupe})
and so $\|f\|_{\BMO}$ is finite.

\begin{definition} The functional space $\BMO$ is defined as the closure 
$$\BMO:= \overline{\left\{ f\in L^\infty + L^2,\ \|f\|_{\BMO}<\infty \right\} }$$
for the $\BMO$ norm.
\end{definition}

\begin{rem} \label{rem:BMO}
The following characterization of the $\BMO$ norm will be useful: for $f\in L^2$ then
\begin{equation}\label{normBMO}
\|f\|_{\BMO}=\sup \limits_{a\textrm{ atom}} |\langle f,a \rangle |
\end{equation}
and $f$ belongs to $\BMO$ if and only if the right hand side is finite.
Indeed if $f\in L^2$ then for all ball $Q$
\begin{align*}
 \mu(Q)^{-\frac12}\|B_Q(f)\|_{L^2(Q)}&=\mu(Q)^{-\frac12} \sup \limits_{\overset{g\in
L^2(Q)}{\|g\|_{L^2(Q)}} \leq1} |\langle B_Q(f),g \rangle | \\
 &=\sup \limits_{\overset{g\in L^2(Q)}{\|g\|_{L^2(Q)}} \leq1}
| \langle f,B_Q(\mu(Q)^{-\frac12}g) \rangle |,
\end{align*}
where we used that $B_Q$ is self-adjoint.
One can easily check that the collection of atoms exactly corresponds to the
collection of functions of type $B_Q(\mu(Q)^{-\frac12}g)$ with $g\in L^2(Q)$ and
$\|g\|_{L^2}\leq 1$.
\end{rem}

Following \cite[Section 8]{BZ}, it comes that $\BMO$ is continuously embedded into
the dual space $(H^1)^*$ and contains $L^\infty$:
$$ L^\infty \hookrightarrow \BMO \hookrightarrow (H^1)^*.$$
Hence
\begin{equation}\label{H1*doncBMO}
 \|T\|_{H^1 \to (H^1)^*} \lesssim \|T\|_{H^1 \to \BMO}.
\end{equation}

The following interpolation theorem between Hardy spaces and Lebesgue spaces is the
key of our study: 
\begin{thm}\label{interpolation}
For all $\theta \in (0,1)$, consider the exponent $p\in(1,2)$ and
$q=p'\in(2,\infty)$ given by
$$\frac 1p=\frac{1-\theta}{2}+\theta \quad \textrm{ and }\quad \frac{1}{q} =
\frac{1-\theta}{2}.$$
Then (using the interpolation notations), we have
$$(L^2, H^1)_{\theta}=L^p \quad \textrm{and} \quad (L^2, (H^1)^*)_{\theta}
\hookrightarrow L^{q},$$
if the ambient space $X$ is non-bounded and
$$ L^p \hookrightarrow L^2+(L^2, H^1)_{\theta} \quad \textrm{and} \quad L^2 \cap
(L^2, (H^1)^*)_{\theta}  \hookrightarrow L^{q},$$
if the space $X$ is bounded.

The same results hold replacing $(H^1)^*$ by $\BMO$ thanks to \eqref{H1*doncBMO}.

\end{thm}

\begin{proof}
The result follows directly from \cite[Theorems 4 and 5]{Ber} (and we keep here its
notations). To ensure that it applies in our setting, we have to check that $H^1
\hookrightarrow L^1$ (which we knew from Corollary \ref{injH1L1}), and that the
maximal function $M_{\infty}$ is bounded by $\mathcal M$, where we recall that
$$M_{\infty}(f)(x) = \sup \limits_{Q\ni x} \ \|A_Q^*(f)\|_{L^{\infty}(Q)},$$ with
$$A_Q=Id-(Id-e^{-r^2H})^M \textrm{ is self-adjoint and }r\textrm{ denotes the radius
of }Q.$$
The binomial theorem shows that $A_Q$ is a linear combination of operators $e^{-kr^2H}$ for $k\in \{1,\ldots,M \}$. 
So the property that $M_{\infty}$ is
pointwisely controlled by $\calM$ is a direct consequence of Proposition
\ref{prop_continuite_semigroupe}.
\end{proof}

In the situation of bounded space (with a finite measure), interpolation is more
delicate since the previous result does not give a complete characterization of
$L^p$ as an intermediate space. We have the following:
\begin{thm}\label{interpolation_finie}  Assume that the space is bounded (or
equivalently that $\mu(X)<+\infty$) and consider a self-adjoint operator $T$
satisfying the following boundedness for some $p\in(1,2)$: 
$$\begin{cases}
 &\|T\|_{L^2 \to L^2} \lesssim 1\\
 &\|T\|_{H^1 \to (H^1)^*} \lesssim A<+\infty\\
 &\|T\|_{L^p  \to L^2} \lesssim B<+\infty,
\end{cases}$$ 
then $T$ is bounded from $L^p$ to $L^{p'}$ with
$$ \|T\|_{L^p \to L^{p'}} \lesssim  B+A^{\frac{1}{p}-\frac{1}{p'}}.$$
The same result holds with $\BMO$ instead of $(H^1)^{*}$ by \eqref{H1*doncBMO}.
\end{thm}

\begin{proof}
 Let $p\in (1,2)$. We aim to apply Theorem \ref{interpolation} to $T$. Pick $\theta
\in (0,1)$ such that $\frac{1- \theta}{2}=1-\frac 1p$. Then $\theta=\frac 1p -
\frac{1}{p'}$.
 Let $f\in L^p \hookrightarrow L^2+(L^2,H^1)_{\theta}$. We chose a decomposition
$f=a+b$ with $a\in L^2$ and $b\in (L^2,H^1)_{\theta}$ such that 
$$ \|a\|_{L^2} + \| b\|_{(L^2,H^1)_{\theta}} \lesssim \|f\|_{L^p}.$$
 Since $T$ is self-adjoint, $T$ is bounded from $ L^2$ to $L^{p'} $ with a norm at
most $B$. Thus 
 $$\|Ta\|_{L^{p'}}\lesssim B\|a\|_{L^2}.$$
 Similarly, by Theorem \ref{interpolation}: 
 $$\|Tb\|_{L^{p'}}\lesssim \|Tb\|_{L^2} + \|Tb\|_{(L^2,(H^1)^*)_{\theta}} \lesssim
B\|b\|_{L^p} + A^{\frac{1}{p}-\frac{1}{p'}} \|b\|_{(L^2,H^1)_{\theta}}.$$
 Moreover $ H^1 \hookrightarrow L^1$ so $(L^2,H^1)_{\theta} \hookrightarrow
(L^2,L^1)_{\theta}= L^p$. Consequently,
  $$\|Tb\|_{L^{p'}}\lesssim  \left( B + A^{\frac{1}{p}-\frac{1}{p'}}\right)
\|b\|_{(L^2,H^1)_{\theta}}.$$
  Hence 
  $$\|Tf\|_{L^{p'}} \lesssim B\|a\|_{L^2} +
\left(B+A^{\frac{1}{p}-\frac{1}{p'}}\right) \|b\|_{(L^2,H^1)_{\theta}}\lesssim
\left(B+A^{\frac{1}{p}-\frac{1}{p'}}\right) \|f\|_{L^p}.$$
\end{proof}

\subsection{On the hypothesis (Hm(A))}

We aim here to study the behaviour of Assumption \eqref{dispmn} with respect to the
parameters $m,n$.

Consider a fix operator $T$, a positive real $A>0$ and let us define property
\eqref{dispmn} for $m\in{\mathbb N}$ and $n>0$: 
\begin{equation}\label{dispmn} \tag{$H_{m,n}(A)$}
      \|T\psi_{m,n}(r^2H)\|_{L^2(B_r)\to L^2(\widetilde{B_r})} \lesssim A
\mu(B_r)^{\frac 12} \mu(\widetilde{B_r})^{\frac12},
  \end{equation}
where $B_r$ and $\widetilde{B_r}$ are any two balls of radius $r>0$.

\begin{prop} \label{prop_hyp_1}
 For all integer $m\geq0$ and $n,n'>0$: 
$$(H_{m,n'}(A)) \Rightarrow (H_{m,n}(A)).$$
\end{prop}

\begin{proof} For simplicity, we deal with the case $n'=1$. 
Assume Property $(H_{m,1}(A))$. Since
$$\psi_{m,n}(x)=x^m e^{-nx}=(nx)^m e^{-nx} n^{-m}=n^{-m} \psi_{m,1}(nx),$$
it comes $$T\psi_{m,n}(r^2H)=n^{-m} T \psi_{m,1}(nr^2H).$$
If $n\geq 1$ then $B_r \subset \sqrt n B_r$ and $\widetilde{B_r} \subset \sqrt
n\widetilde{B_r}$.
%
Hence, using the doubling property we get
\begin{align*}
 \|T\psi_{m,n}(r^2H)\|_{L^2(B_r)\to L^2(\widetilde{B_r})}&=n^{-m}
\|T\psi_{m,1}(nr^2H)\|_{L^2(B_r)\to L^2(\widetilde{B_r})} \\
 &\leq n^{-m} \|T\psi_{m,1}(nr^2H)\|_{L^2(\sqrt n B_r)\to L^2(\sqrt n
\widetilde{B_r})} \\
 &\leq n^{-m} A \mu(\sqrt n B_r)^{\frac12} \mu(\sqrt n
\widetilde{B_r})^{\frac12}\lesssim A \mu(B_r)^{\frac 12}
\mu(\widetilde{B_r})^{\frac12}.
\end{align*}
If $n \leq 1$ then $\sqrt{n} \widetilde{B_r} \subset \widetilde{B_r}$. We cover
$\widetilde{B_r}$ by $N\simeq \left( \frac{r}{\sqrt n r}\right)^d=n^{-\frac d2}$
balls $\widetilde{B_j}$ of radius $\sqrt n r$ and $B_r$ by $N$ balls $B_k$ of radius
$\sqrt nr$ (satisfying the bounded overlap property).
Thus $$\|T(\psi_{m,n}(r^2H)f)\|_{L^2(\widetilde{B_r})} \leq \sum_j \sum_k
\|T(\psi_{m,n}(r^2H)f.\mathbf1_{B_k})\|_{L^2(\widetilde{B_j})}.$$

Finally: 
\begin{align*}
 \|T\psi_{m,n}(r^2H)\|_{L^2(B_r) \to L^2(\widetilde{B_r})} &\leq \sum_j \sum_k
n^{-m} \|T_t(H) \psi_{m,1}(nr^2H)\|_{L^2(B_k) \to L^2(\widetilde{B_j})}\\
 &\lesssim \sum_j \sum_k n^{-m} A \mu(B_k)^{\frac 12} \mu(\widetilde{B_j})^{\frac12} \\
 &\lesssim A \left(\sum_j \mu(\widetilde{B_j})\right)^{\frac12} \left(\sum_k
\mu(B_k)\right)^{\frac12}  \\
 &\lesssim A \mu(B_r)^{\frac 12} \mu(\widetilde{B_r})^{\frac 12}.
\end{align*}
\end{proof}
We will now be able to focus on $(H_{m,1}(A))$ and functions $\psi_{m,1}=\psi_m$ rather
than keeping the dependence in the parameter $n$.
%
\begin{prop}\label{prop_hyp_2}
 If $m'>m\geq0$ are two integers then $$(H_{m,1}(A)) \Rightarrow (H_{m',1}(A)).$$
\end{prop}

\begin{proof}
 Assume $(H_{m,1}(A))$ is satisfied. Then, by Proposition \ref{prop_hyp_1}, $(H_{m,n}(A))$
is also true for all $n>0$. 
First we remark that $$ T\psi_{m',1}(r^2H)=T\psi_{m,\frac 12}(r^2H)\psi_{m'-m,\frac
12}(r^2H).$$
Hence, decomposing $X$ in dyadic coronas around $B_r$: 
\begin{align*}
 \|T&\psi_{m',1}(r^2H)\|_{L^2(B_r) \to L^2(\widetilde{B_r})} \\
 &\leq \sum_{j=0}^{+\infty} \|T\psi_{m,\frac 12}(r^2H)\|_{L^2(C_j(B_r)) \to
L^2(\widetilde{B_r})} \|\psi_{m'-m,\frac 12}(r^2H)\|_{L^2(B_r) \to L^2(C_j(B_r))}.
\end{align*}
Let $f \in L^2(B_r)$. We treat the case $j=0$ with Proposition
\ref{prop_continuite_semigroupe} and $(H_{m,\frac 12}(A))$: 
$$\|T\psi_{m,\frac 12}(r^2H)\|_{L^2(B_r) \to L^2(\widetilde{B_r})}
\|\psi_{m'-m,\frac 12}(r^2H)\|_{L^2(B_r) \to L^2(B_r)}\lesssim A \mu(B_r)^{\frac 12}
\mu(\widetilde{B_r})^{\frac 12}.$$
Assume now that $j\geq1$. If $x\in C_j(B_r)$, then by Cauchy-Schwarz inequality:  
\begin{align*}
 |\psi_{m'-m,\frac 12}(r^2H)f(x)|&\leq \int_{B_r}
\frac{1}{\mu(B(x,r))}e^{-c\frac{d(x,y)^2}{r^2}}|f(y)|d\mu(y) \\
 &\leq \frac{e^{-c2^{2j}} \mu(B_r)^{\frac 12}}{\mu(B(x,r))} \|f\|_{L^2(B_r)}.
\end{align*}
By \eqref{ball_behaviour}, we have already seen that for every $x\in C_j(B_r)$
$$ \mu(B_r) \lesssim 2^{jd} \mu(B(x,r)),$$
which yields
\begin{align*}
 |\psi_{m'-m,\frac 12}(r^2H)f(x)| \lesssim e^{-c2^{2j}} 2^{jd} \mu(B_r)^{-\frac 12}
\|f\|_{L^2(B_r)}.
\end{align*}
Hence, by the doubling property:
\begin{equation} \|\psi_{m'-m,\frac 12}(r^2H)f\|_{L^2(B_r) \to L^2(C_j(B_r))}
\lesssim e^{-c2^{2j}}2^{\frac{3jd}{2}}. \label{eq:11} \end{equation}
Consider $(B_{k})_{k=0,...,K}$ a collection of balls of radius $r$ (with a bounded
overlap property so $K\lesssim 2^{jd}$) which covers $C_j(B_r)$ with, by the
doubling property: $\mu(B_k) \lesssim 2^{jd} \mu(B_r)$.
From $(H_{m,\frac 12})$ it follows
\begin{align}
\|T\psi_{m,\frac 12}(r^2H)\|_{L^2(C_j(B_r)) \to L^2(\widetilde{B_r})} & \leq
\sum_{k=0}^N \|T \psi_{m,\frac 12}(r^2H)\|_{L^2(B_k) \to L^2(\widetilde{B_r})}
\nonumber \\
& \leq \sum_{k=0}^N A \mu(B_{k})^{\frac 12} \mu(\widetilde{B_r})^{\frac 12}
\nonumber \\
 &  \lesssim A 2^{\frac{3}{2}jd} \mu(B_{r})^{\frac 12} \mu(\widetilde{B_r})^{\frac
12}. \label{eq:12}
\end{align}
Thus, combining (\ref{eq:11}) and (\ref{eq:12}), it comes 
\begin{align*}
 \|T\psi_{m',1}(r^2H)\|_{L^2(B_r) \to L^2(\widetilde{B_r})} &\lesssim
\left(1+\sum_{j\geq1} e^{-c2^{2j}}2^{3jd}\right) A \mu(B_r)^{\frac 12}
\mu(\widetilde{B_r})^{\frac 12} \\
 &\lesssim A \mu(B_r)^{\frac 12} \mu(\widetilde{B_r})^{\frac 12},
\end{align*}
which ends the proof of Property $(H_{m',1}(A))$.
\end{proof}

We sum up Propositions \ref{prop_hyp_1} and \ref{prop_hyp_2} in:

\begin{thm} \label{thm:H} Assume \eqref{dd}. For every integer $m\geq 0$, Property
$(H_{m,n}(A))$ is independent on $n>0$. So let us call $(H_m(A))$ this property. It is
``increasing in $m$'', since for two integers $m'>m\geq 0$ 
$$ (H_{m}(A)) \Rightarrow (H_{m'}(A)).$$
\end{thm}

\section{Dispersion inequality from Property (Hm(A))} \label{section_dispersion}

The aim of this section is to show Theorem \ref{thm1}, more precisely that Property
\eqref{disp} implies a $H ^1-BMO$ and $L^p-L^{p'}$ dispersive estimates.
The main idea is first to prove boundedness of the operator on atoms, then to deduce
boundedness on the whole Hardy space $H^1$, and finally to interpolate with the
$L^2$-boundedness.

In all this section, we fix a large enough integer $M\geq\min(3, \frac{3}{4}+ \frac{3d}{8})$,
which allows us to consider the notions of {\it atoms} and Hardy spaces $H^1$, built
with this parameter. As pointed out in Remark \ref{rem:M}, that also allows us to
find an integer $m \in [\frac d2,\frac{4M}{3}]$.

\subsection{Boundedness on atoms}

\begin{thm}\label{thm_atome_1} Assume \eqref{dd} and \eqref{due}. Let $T$ be a
$L^2$-bounded operator, which commutes with $H$.
 If $T$ satisfies Property \eqref{disp} for a certain integer $m \leq \frac{4M}{3}$, then
one gets 
$$ \sup_{a,b} |\langle T a,b\rangle | \lesssim A,$$
where the supremum is taken over all atoms $a$, $b$.
\end{thm}

\begin{proof}
Let $a$ and $b$ be two atoms. By definition, there exists $B_1$ and $B_2$ two balls
with radiuses $r_1$ and $r_2$ respectively, and $f\in L^{2}(B_1) \virg g\in
L^{2}(B_2)$, such that
$$\begin{cases}
 a=(1-e^{-r_1^2H})^Mf \quad \textrm{with } \|f\|_{L^{2}(B_1)} \leq \mu(B_1)^{-\frac
12}\\
 b=(1-e^{-r_2^2H})^Mg \quad \textrm{with } \|g\|_{L^{2}(B_2)} \leq \mu(B_2)^{-\frac
12}\\
\end{cases}.$$
We first remark by (c) of proposition \ref{prop_basique} that: 
\begin{align*}
 a&=\left ( \int_0^{r_1^2} He^{-sH}ds \right)^M f
=\int_0^{r_1^2}\ldots\int_0^{r_1^2} H^Me^{-(s_1+\ldots+s_M)H}fds_1\ldots ds_M \\
 &=\int_0^{Mr_1^2} \underbrace{\left(\int_{\underset{0\leq s_i \leq
r_1^2}{s_1+\ldots +s_M=u}}ds_1 \ldots ds_{M-1} \right)}_{=I_M(u)} H^M e^{-uH}fdu.
\end{align*}
As $s_i \geq 0$ for all $i\in \{1, \ldots ,M\}$ with $s_1+\cdots +s_M=u$, we have: 
$0\leq s_i \leq u$. \\
Hence $$I_M(u) \leq u^{M-1}.$$ 
Thus $$\< Ta,b\>=\int_0^{Mr_1^2} \int_0^{Mr_2^2} I_M(u) I_M(v) \< T
\psi_{M}(uH)f,\psi_{M}(vH)g \>\frac{dv}{v^M} \frac{du}{u^M}.$$
Moreover $\psi_{M}$ is continuous and $H$ is self-adjoint, so $\psi_{M}(uH)$ and
$\psi_{M}(vH)$ are also self-adjoint.
Using (a) and (b) of Proposition \ref{prop_basique} and the fact that $T$ commutes 
with $H$ (and so with every operator $\psi_{m,n}(H)$), we get: 
\begin{align*}
 |\< T a,b\>|&\leq \int_0^{Mr_1^2} \int_0^{Mr_2^2}
|\<T_t(H)\psi_{M,1}(uH)\psi_{\frac M3, \frac 13}(vH) \psi_{\frac M3, \frac
13}(vH)f,\psi_{\frac M3, \frac 13}(vh)g\>|\frac{du}{u} \frac{dv}{v} \\
 &= \iint |\<T_t(H)\frac{(uv^{\frac 13})^M}{(u+\frac v3)^{\frac {4M}3}} \psi_{\frac
{4M}3, 1}((u+\frac v3 )H)\psi_{\frac M3,\frac 13}(vH)f,\psi_{\frac M3, \frac
13}(vh)g\>|\frac{du}u \frac{dv}v.
\end{align*}
Here we have decomposed $\psi_{M,1}$ in three terms involving $\psi_{\frac M3, \frac
13}$. We aim to use in particular the Gaussian estimates \eqref{UEmn},
which hold only if $\frac{M}{3}$ is an integer. We should decompose $M=M_1+M_2+M_3$
with $3$ integers $M_1,M_2,M_3$ which are comparable to $M/3$ (that is why we picked $M\geq3$). For simplicity we
take $M_1=M_2=M_3=M/3$ and assume that they are integers. We let to the reader the
minor modifications.

Without losing generality because the problem is symmetric in $u$ and $v$, we can
assume that $u \leq v$ so that $ \frac v3 \leq u+\frac v3 \leq \frac{4v}{3} $. Hence
$\frac{uv^{\frac13}}{(u+\frac v3)^{\frac43}}\simeq \frac uv$.
We cover the whole space $X$ by balls $B_j$ and $B_k$ of radius $\sqrt{u+\frac v3}$.
The covering satisfies the bounded overlap property. We use Cauchy-Schwarz
inequality and Property \eqref{disp} to obtain: 
\begin{align*} 
 &|\< Ta,b\>| \leq \\
 &\leq \iint\sum_{j, k} \left(\frac uv \right)^M |\<\mathbf1_{B_k}T \psi_{\frac
{4M}3, 1}((u+\frac v3 )H)\mathbf1_{B_j}\psi_{\frac M3,
\frac13}(vH)f,\mathbf1_{B_k}\psi_{\frac M3, \frac 13}(vH)g\>|\frac{du}u \frac{dv}v
\\
 &\lesssim \iint\sum_{j, k} \left(\frac{u}{v}\right)^M \|T \psi_{\frac {4M}3,
1}((u+\frac v3 )H)\mathbf1_{B_j}\psi_{\frac M3, \frac13}(vH)f\|_{L^2(B_k)}
\|\psi_{\frac M3, \frac 13}(vH)g\|_{L^2(B_k)}\frac{du}u \frac{dv}v \\
 &\lesssim \iint\sum_{j, k} \left(\frac{u}{v}\right)^M A \mu(B_k)^{\frac 12}
\mu(B_j)^{\frac 12}\|\psi_{\frac M3, \frac13}(vH)f\|_{L^2(B_j)} \|\psi_{\frac M3,
\frac 13}(vH)g\|_{L^2(B_k)}\frac{du}u \frac{dv}v,
\end{align*}
where we have used that $T$ satisfies Property $(H_{4M/3})$. Indeed $T$ satisfies
property \eqref{disp} for $m\leq4M/3$ (so $T$ satisfies also $(H_{4M/3})$ by Theorem
\ref{thm:H}).
To simplify the notation we will now note $\psi_{\frac M3, \frac 13}=\psi$. We use a
decomposition in dyadic coronas around $B_1$: 
$$ \sum_{j\in J} \mu(B_j)^{\frac 12} \|\psi(vH)f\|_{L^2(B_j)}\leq \sum_{j \in J}
\sum_{l=0}^{+\infty} \mu(B_j)^{\frac 12} \|\mathbf1_{C_l(B_1)}
\psi(vH)f\|_{L^2(B_j)}.$$
We study the terms $l=0$ and $l \geq 1$ separately. \\
First when $l=0$:
\begin{align*}
 \sum_{j\in J} \mu(B_j)^{\frac 12} \|\mathbf1_{C_0(B_1)}\psi(vH)f\|_{L^2(B_j)} &\leq
\left( \sum_J \mu(B_j) \right)^{\frac12} \left( \sum_J \|\mathbf1_{B_1}
\psi(vH)f\|_{L^2(B_j)}^2 \right)^{\frac 12} \\
  &\lesssim \mu(B_1)^{\frac12} \left( \sum_J \int_{B_j} |\mathbf1_{B_1}(x)
\psi(vH)f(x)|^2 d\mu(x) \right)^{\frac12} \\
  &\lesssim \mu(B_1)^{\frac12} \|\psi(vH)f\|_{L^2}.
\end{align*}
Now when $l\geq 1$ the number of indices in $J$ for which the sum is nonzero is
equivalent to the number of balls of radius $\sqrt{u+\frac v3}$ we need to cover
$C_l(B_1)$, that is $|J| \simeq \left( \frac{2^lr_1}{\sqrt{u+\frac v3}} \right)^d$. 

Now, remark that by the doubling property of the measure and \eqref{ball_behaviour},
since $B_j$ is a ball of radius $\sqrt{u+v/3} \simeq \sqrt{v}$ we deduce that for
$x\in B_j \cap C_l(B_1)$ then 
$$\mu(B(x,\sqrt v)) \simeq \mu(B_j).$$ 

By \eqref{UEmn}, we have:
\begin{align*}
 \sum_{j\in J} \sum_{l=1}^{+\infty} \mu(B_j)^{\frac 12} \|\mathbf1_{C_l(B_1)}
\psi(vH)f\|_{L^2(B_j)} & \lesssim \sum_{j\in J} \sum_{l=1}^{+\infty}
\mu(B_j)^{\frac 12}\left \|\frac{\mathbf1_{C_l(B_1)}(x)}{\mu(B(x,\sqrt
v))}e^{-\frac{2^{2l}r_1^2}{v}} \|f\|_{L^1(B_1)}\right \|_{L^2_x(B_j)} \\
 &\lesssim \sum_{j\in J} \sum_{l=1}^{+\infty} \mu(B_j)^{\frac 12} \mu(B_j)^{\frac
12} \frac{1}{\mu(B_j)}e^{-\frac{2^{2l}r_1^2}{v}} \\
 &\lesssim \sum_{l\geq 1} \left( \frac{2^lr_1}{\sqrt{u+\frac v3}}
\right)^de^{-\frac{2^{2l}r_1^2}{v}} \lesssim \sum_{l\geq 1}\left(\frac{2^l
r_1}{\sqrt v} \right)^d e^{-\frac{2^{2l}r_1^2}{v}},
\end{align*}
where we have used the $L^2$-normalization of $f$, which yields that $\|f\|_{L^1} \lesssim 1$.

We then refer the reader to Lemma \ref{lem_calcul} to estimate the sum and it comes
$$\sum_{j\in J} \sum_{l=1}^{+\infty} \mu(B_j)^{\frac12}\|\chi_{C_l(B_1)}
\psi(vH)f\|_{L^2(B_j)} \lesssim \left( \frac{r_1}{\sqrt v}\right)^{-1}.$$
Thus $$\sum_{j\in J}\mu(B_j)^{\frac12}\|\psi(vH)f\|_{L^2(B_j)} \leq
\mu(B_1)^{\frac12} \|\psi(vH)f\|_{L^2} + \frac{\sqrt v}{r_1}.$$
Similarly for $B_2$ and the sum over $k \in K$: $$\sum_{k\in
K}\mu(B_k)^{\frac12}\|\psi(vH)g\|_{L^2(B_k)} \leq \mu(B_2)^{\frac12}
\|\psi(vH)g\|_{L^2} + \frac{\sqrt v}{r_2}.$$
Hence, one concludes: 
\begin{align*}
\lefteqn{| \langle T a,b \rangle |} & & \\
& &  \lesssim A \iint &\left(\frac uv \right)^M \left( \mu(B_1)^{\frac12}
\|\psi(vH)f\|_{L^2} + \frac{\sqrt v}{r_1} \right) \left( \mu(B_2)^{\frac12}
\|\psi(vH)g\|_{L^2} + \frac{\sqrt v}{r_2} \right) \frac{du}{u}\frac{dv}{v}.
\end{align*}
We then develop the product to split the problem into four different terms: 
\begin{align*}
I& =\iint \left( \frac uv \right)^M \mu(B_1)^{\frac 12}
\mu(B_2)^{\frac12}\|\psi(vH)f\|_{L^2} \|\psi(vH)g\|_{L^2} \frac{du}{u}\frac{dv}{v},
\\
II& =\iint \left( \frac uv \right)^M \mu(B_1)^{\frac12}\|\psi(vH)f\|_{L^2}
\frac{\sqrt v}{r_2} \frac{du}{u}\frac{dv}{v},\\
III& =\iint \left( \frac uv \right)^M \mu(B_2)^{\frac12}\|\psi(vH)g\|_{L^2}
\frac{\sqrt v}{r_1} \frac{du}{u}\frac{dv}{v},\\
IV& =\iint \left( \frac uv \right)^M \frac{\sqrt v}{r_1}\frac{\sqrt v}{r_2}
\frac{du}{u}\frac{dv}{v}. 
\end{align*}

We discern now two cases: 

\medskip
\noindent
{\bf Case 1:} $0\leq u \leq v \leq R=\min(Mr_1^2,Mr_2^2)$. 

Then Item (d) of Proposition \ref{prop_basique} yields
\begin{align*}
 I&= \int_{v=0}^{R} \int_{u=0}^{v} \left( \frac uv \right)^M \mu(B_1)^{\frac 12}
\mu(B_2)^{\frac12}\|\psi(vH)f\|_{L^2} \|\psi(vH)g\|_{L^2} \frac{du}{u}\frac{dv}{v}
\\
 &=\mu(B_1)^{\frac 12} \mu(B_2)^{\frac12} \int_0^R \|\psi(vH)f\|_{L^2}
\|\psi(vH)g\|_{L^2} \frac{dv}{v} \\
 &\leq \mu(B_1)^{\frac 12} \mu(B_2)^{\frac12} \left(\int_0^{+\infty}
\|\psi(vH)f\|_{L^2}^2 \frac{dv}{v}\right)^{\frac 12}
\left(\int_0^{+\infty}\|\psi(vH)g\|_{L^2}^2 \frac{dv}{v}\right)^{\frac12} \\
 &\leq \mu(B_1)^{\frac 12}\|f\|_{L^2} \mu(B_2)^{\frac12}\|g\|_{L^2} \lesssim 1,
\end{align*}
Similarly for the second term, 
\begin{align*}
 II=& \mu(B_1)^{\frac 12} \frac{1}{r_2}\int_0^R \|\psi(vH)f\|_{L^2} \sqrt v
\frac{dv}{v} \\
 &\leq \mu(B_1)^{\frac 12} \frac{1}{r_2}\|f\|_{L^2} \left(\int_0^R v
\frac{dv}{v}\right)^{\frac 12}\leq \frac{R^{\frac12}}{r_2}\leq
\frac{\sqrt{Mr_2^2}}{r_2}\lesssim 1.
\end{align*}
The third term is treated the same way.
The fourth term gives: $$IV=\frac{1}{r_1r_2}\int_0^R \sqrt v \sqrt v \frac{dv}{v} =
\frac{R}{r_1r_2} \leq
\frac{\min(Mr_1^2,Mr_2^2)}{\sqrt{\min(r_1^2,r_2^2)}\sqrt{\min(r_1^2,r_2^2)}}\lesssim
1.$$
So in this first case, we obtain
\begin{equation}
 I+II+III+IV \lesssim 1. \label{case1}
\end{equation}

\medskip
\noindent
{\bf Case 2:} $0\leq u \leq Mr_1^2 \leq v \leq Mr_2^2$.

Similarly we get: 
\begin{align*}
 I&= \mu(B_1)^{\frac 12} \mu(B_2)^{\frac12} \int_{v=Mr_1^2}^{Mr_2^2}
\int_{u=0}^{Mr_1^2} \left(\frac uv \right)^M \|\psi(vH)f\|_{L^2}
\|\psi(vH)g\|_{L^2} \frac{du}{u} \frac{dv}{v} \\
 &\leq \mu(B_1)^{\frac 12} \mu(B_2)^{\frac12} \int_{v=Mr_1^2}^{Mr_2^2}
\int_{u=0}^{Mr_1^2} \frac{uv^{M-1}}{v^M} \|\psi(vH)f\|_{L^2} \|\psi(vH)g\|_{L^2}
\frac{du}{u} \frac{dv}{v} \\
 &=\mu(B_1)^{\frac 12} \mu(B_2)^{\frac12} \int_{v=Mr_1^2}^{Mr_2^2} \frac{Mr_1^2}{v}
\|\psi(vH)f\|_{L^2} \|\psi(vH)g\|_{L^2} \frac{dv}{v} \\
 &\leq \mu(B_1)^{\frac12}\mu(B_2)^{\frac12} Mr_1^2 \sup \limits_{v
\in[Mr_1^2,Mr_2^2]} \frac1v \left( \int_0^{Mr_2^2} \|\psi(vH)f\|_{L^2}^2
\frac{dv}{v} \right)^{\frac12} \left( \int_0^{Mr_2^2} \|\psi(vH)g\|_{L^2}^2
\frac{dv}{v} \right)^{\frac 12} \\
 &\leq \frac{Mr_1^2}{Mr_1^2} = 1.
\end{align*}
For the second term, since $r_1 \leq r_2$:$$ II\leq \mu(B_1)^{\frac 12}
\frac{Mr_1^2}{r_2}\left( \int_0^{Mr_1^2} \|\psi(vH)f\|_{L^2}^2 \frac{dv}{v}
\right)^{\frac12} \left( \int_{Mr_1^2}^{+\infty}\frac{1}{v}\frac{dv}{v}
\right)^{\frac12}\lesssim \frac{r_1^2}{r_2r_1} \leq 1.$$
The third term is similar: $$III\leq
\mu(B_2)^{\frac12}Mr_1^2\int_{Mr_1^2}^{Mr_2^2}\frac{1}{v}\|\psi(vH)g\|_{L^2}
\frac{\sqrt v}{r_1}\frac{dv}{v}\lesssim r_1 \left( \int_{Mr_1^2}^{+\infty}
\frac{1}{v}\frac{dv}{v} \right)^{\frac12} \lesssim \frac{r_1}{r_1}=1.$$
Finally we treat the last term: $$ IV\leq
\int_{Mr_1^2}^{Mr_2^2}\frac{Mr_1^2}{v}\frac{\sqrt v}{r_1}\frac{\sqrt
v}{r_2}\frac{dv}{v}\lesssim
\frac{r_1}{r_2}\int_{Mr_1^2}^{Mr_2^2}\frac{1}{v}\frac{dv}{v}=2\frac{r_1}{r_2}\ln\left(
\frac{r_2}{r_1} \right)\lesssim1,$$
because $x\mapsto \frac{\ln x}{x}$ is continuous if $x\geq1$, equals $0$ if $x=1$
and tends to $0$ as $x$ tends to $+\infty$, so is bounded uniformly in $x\geq 1$
(here $\frac{r_2}{r_1}\geq 1$).

Thus, in this second case, we also conclude that
\begin{equation}
 I+II+III+IV \lesssim 1. \label{case2}
\end{equation}

Since $u\leq v$ (which was assumed at the beginning by symmetry), cases 1 and 2
cover all the possible situations.
Consequently, we deduce that for all atoms $a$ and $b$, we have 
$$|\<Ta,b\>|\lesssim A,$$
where the implicit constant does not depend on the atoms (but maybe on the
parameters $M$ and $m$).
\end{proof}

%

We used the following lemma with $N=1$ and $x=\frac{r}{\sqrt v}$:
\begin{lem}\label{lem_calcul}
 Let $x>0$ and $d \in \mathbb N$. For all $N \in \mathbb N^*$: 
 $$\sum_{l=0}^{+\infty} (2^lx)^d
e^{-(2^lx)^2} \lesssim x^{-N}.$$
\end{lem}

\begin{proof}
We remark that
$\int_{2^l}^{2^{l+1}}\frac{dt}{t}=\ln\left(\frac{2^{l+1}}{2^l}\right)=\ln2$. Thus: 
$$\sum_{l=0}^{+\infty}(2^lx)^d e^{-(2^lx)^2}
= \frac{1}{\ln2} \sum_{l=0}^{+\infty }(2^lx)^d
e^{-(2^lx)^2}\int_{2^l}^{2^{l+1}}\frac{dt}{t}.$$ 
$2^l \leq t \leq 2^{l+1}$ yields $(2^lx)^d\leq
(tx)^d$ and $e^{-(tx)^2}\geq e^{-(2^{l+1}x)^2}$. So we have: $e^{-(2^lx)^2}\leq e^{-\frac{(tx)^2}{4}}$.
Hence: 
\begin{align*}
 \sum_{l=0}^{+\infty} ( 2^lx)^d e^{-(2^lx)^2}
& \lesssim \int_1^{+\infty} (tx)^d e^{-\frac{(tx)^2}{4}}\frac{dt}{t} =\int_{\frac{x}{2}}^{+\infty}
(2u)^d e^{-u^2} \frac{du}{u} \\
 &\lesssim \int_{\frac{x}{2}}^{+\infty} \frac{1}{u^N} \frac{du}{u} = \left[
\frac{u^{-N}}{-N}\right]_{\frac{x}{2}}^{+\infty} \lesssim x^{-N}
\end{align*}
for $N\in \mathbb N^*$ as large as we want.
\end{proof}

\subsection{Boundedness on Hardy space}

After having proved that the operator $T$ (of Theorem \ref{thm1}) is bounded on
atoms, we now aim to show that $T$ is bounded from the Hardy space $H^1$ to its dual
$(H^1)^*$ (and more precisely to $\BMO$) with a norm controlled by $A$. 
If $f\in H^1$ then there exists an atomic decomposition $f=\sum_{i=0}^{+\infty}
\lambda_i a_i$ where $a_i$ are atoms and $\sum_{i=0}^{+\infty} |\lambda_i| <
2\|f\|_{H^1}$.
We know how to bound the operator on atoms, we would like to extend it passing to
the limit in 
$$T \left(\sum_{i=0}^{N} \lambda_i a_i\right)=\sum_{i=0}^N \lambda_i T a_i$$ 
in order to apply Theorem \ref{thm_atome_1}.
As $N$ goes to $+\infty$, that last equality may not be true. Indeed, one can find
in \cite{Bow} an example (due to Meyer) of a linear form bounded on atoms, which
is not bounded on the whole Hardy space. So to rigorously check this step, we need
to prove it using specificities of our situation. Aiming that, we are going to use
an approximation of the identity well suited to our frame: $(e^{-sH})_{s>0}$.\\

We start by showing that $Te^{-sH}$ (the regularized version of $T$) satisfies the
same estimate as the one in Theorem \ref{thm_atome_1}: 
\begin{thm}\label{thm_atome_2} Assume \eqref{dd} and \eqref{due}. Consider a fixed
operator $T$, $L^2$-bounded, commuting with $H$ and satisfying Property
\eqref{disp} for some integer $m \in [\frac d2,\frac{4M}{3}]$. Then
uniformly with respect to $s>0$, the operator $T e^{-sH}$ satisfies Property
\eqref{disp} and so by Theorem \ref{thm_atome_1}:
$$ \sup_{s>0} \ \sup_{a,b} |\<Te^{-sH}a,b\>| \lesssim A,$$
where the supremum is taken over all the atoms $a,b$.
\end{thm}

\begin{proof}
Set $U_{s}:=Te^{-sH}$. It suffices to check that $U_{s}$ satisfies Property
\eqref{disp} uniformly in $s$, which is 
$$\|U_{s}\psi_{m,1}(r^2H)\|_{L^2(B_r)\to L^2(\widetilde{B_r})} \lesssim t^{-\frac
d2} \mu(B_r)^{\frac 12} \mu(\widetilde{B_r})^{\frac 12},$$ for any two balls $B_r$
and $\widetilde{B_r}$ with radius $r>0$. 
First, remark that 
$$U_{s}\psi_{m,1}(r^2H)=Te^{-sH}(r^2H)^me^{-r^2H}=T \psi_{m,1}((r^2+s)H)\left(
\frac{r^2}{r^2+s} \right)^m,$$
so that
$$\|U_{s} \psi_{m,1}(r^2H)\|_{L^2(B_r)\to L^2(\widetilde{B_r})}= \left(
\frac{r^2}{r^2+s} \right)^m \|T\psi_{m,1}((r^2+s)H)\|_{L^2(B_r)\to
L^2(\widetilde{B_r})}.$$
As $r^2<r^2+s$, the balls of radius $r$ are included into the balls with same
centers and radius $\sqrt{r^2+s}$ denoted $B_{\sqrt{r^2+s}}$ and
$\widetilde{B_{\sqrt{r^2+s}}}$. 
Then it comes (with Property \eqref{disp} for $T$ and the doubling property)
\begin{align*}
 \|U_{s} \psi_{m,1}(r^2H)\|_{L^2(B_r)\to L^2(\widetilde{B_r})}& \leq \left(
\frac{r^2}{r^2+s} \right)^m \|T\psi_{m,1}((r^2+s)H)\|_{L^2(B_{\sqrt{r^2+s}})\to
L^2(\widetilde{B_{\sqrt{r^2+s}}})} \\
 &\leq \left( \frac{r^2}{r^2+s} \right)^m A \mu(B_{\sqrt{r^2+s}})^{\frac 12}
\mu(\widetilde{B_{\sqrt{r^2+s}}})^{\frac 12} \\
 &\lesssim \left( \frac{r^2}{r^2+s} \right)^m A \sqrt{\frac{r^2+s}{r^2}}^{\frac
d2}\mu(B_r)^{\frac 12} \sqrt{\frac{r^2+s}{r^2}}^{\frac
d2}\mu(\widetilde{B_r})^{\frac 12}\\
 &\leq A \left(\frac{r^2}{r^2+s}\right)^{m-\frac d2} \mu(B_r)^{\frac
12}\mu(\widetilde{B_r})^{\frac 12}\leq A \mu(B_r)^{\frac 12}
\mu(\widetilde{B_r})^{\frac 12},
\end{align*}
where the last inequality comes from $m \geq \frac d2$. That concludes the proof of
Property \eqref{disp} for the operator $U_s$ and all the estimates are uniform with
respect to $s>0$.
\end{proof}

In order to prove that we can pass to the limit as $N$ goes to $+\infty$ in
$$Te^{-sH}\left(\sum_{i=0}^{N} \lambda_i a_i\right)=\sum_{i=0}^N \lambda_i
Te^{-sH}a_i$$ for atoms $a_i$, we have to show some continuity of the operator
$Te^{-sH}$. 
\begin{thm}\label{thm_continuite}
 If $T$ is a $L^2$-bounded operator which commutes with $H$ and the ambient space
$X$ satisfies the uniform control of the volume \eqref{d}, then for all $s>0$: $T
e^{-sH}$ maps $L^1$ to $L^{\infty}$ and
$$ \|Te^{-sH}\|_{L^1 \to L^\infty} \lesssim s^{-\frac \nu 2}.$$
\end{thm}

\begin{proof}
By the commutativity property, we write $T e^{-sH} = e^{-sH/2} T e^{-sH/2}$. Hence
$$ \|T e^{-sH}\|_{L^1 \to L^{\infty}} \leq \|e^{-\frac s2 H}\|_{L^1 \to L^2}
\|T\|_{L^2 \to L^2} \|e^{-\frac s2 H}\|_{L^2 \to L^{\infty}} .$$
Using the Gaussian pointwise estimates of the heat kernel and \eqref{d}, we deduce
by a $T^*T$ argument that
\begin{align*}
 \|e^{-\frac s2 H}\|_{L^1 \to L^{2}}^2 & = \| e^{-sH} \|_{L^1 \to L^\infty} \\
 & = \sup_{x,y} p_{s}(x,y) \lesssim s^{-\frac \nu 2}
\end{align*}
and by duality 
\begin{align*}
 \|e^{-\frac s2 H}\|_{L^1 \to L^{2}}= \|e^{-\frac s2 H}\|_{L^2 \to L^{\infty}}
\lesssim s^{-\frac \nu 4}.
\end{align*}
As a consequence, we deduce the desired estimate.
\end{proof}

We are now able to establish the result on the whole Hardy space $H^1$: 
\begin{thm} \label{thm_hardy_1} Assume \eqref{dd}, \eqref{d} and \eqref{due}.
Consider a $L^2$-bounded operator $T$, which commutes with $H$ and which satisfies
Property \eqref{disp} for some integer $m\in[\frac d2,\frac{4M}{3}]$.
Then $T$ and $T e^{-sH}$, for all $s>0$, can be continuously extended from $H^1$ to
$BMO$ (and so in particular to its dual $(H^1)^*$) and we have 
$$ \|T \|_{H^1 \to BMO} + \sup_{s>0} \|T e^{-sH} \|_{H^1 \to BMO} \lesssim A.$$
\end{thm}

\begin{proof}
Let $f\in H^1$ and consider an atomic decomposition. The atoms are uniformly bounded
in $L^1$ so the limit $$f=\sum_{i=0}^{+\infty}\lambda_i a_i=\lim
\limits_{N\to+\infty} \sum_{i=0}^N \lambda_i a_i$$ takes place in $L^1$. \\
Moreover $a_i \in L^1$ implies $Te^{-sH}(a_i) \in L^{\infty}$ due to Theorem
\ref{thm_continuite}. Hence the limit
$$ Te^{-sH}\left(\lim \limits_{N \to +\infty} \sum_{i=0}^{N} \lambda_i a_i \right)
=\lim \limits_{N \to + \infty} Te^{-sH}\left(\sum_{i=0}^{N} \lambda_i a_i \right)
=\lim \limits_{N \to + \infty} \sum_{i=0}^N \lambda_i T e^{-sH}(a_i)$$
is valid and takes place in $L^{\infty}$ for every $s>0$ fixed.
Thus 
$$Te^{-sH} \left( f\right) = \sum_{i=0}^{+\infty} \lambda_i T e^{-sH}(a_i).$$
Let $f\in H^1$. There exists a decomposition $f=\sum_i \lambda_i a_i$ with $a_i$
atoms, $\sum_i |\lambda_i|<+\infty$ and $\sum_i |\lambda_i| \leq 2 \|f\|_{H^1}$. We
want to estimate $$\|Te^{-sH}f\|_{\BMO}=\sup \limits_{b}|\langle Te^{-sH}f,b \rangle |$$
where the supremum is taken over all atoms $b$ (see Remark \ref{rem:BMO}). By
Theorem \ref{thm_atome_2}, and what we just prove, we have:
\begin{align*}
 | \langle Te^{-sH}\sum_i \lambda_i a_i,b \rangle |&\leq \sum_i|\lambda_i| | \langle Te^{-sH}a_i,b \rangle |\\
 &\lesssim \sum_i|\lambda_i| A\lesssim A\|f\|_{H^1}.
\end{align*}
Hence $$\|Te^{-sH}\|_{H^1\to \BMO} \lesssim A$$ and the implicit constant is uniform
in $s>0$.


Let us now consider the boundedness of the operator $T$.
We know (see \cite{BZ} e.g.) that $H^1 \cap L^2$ is dense in $H^1$ (since every
atoms are $L^2$ functions). Moreover $(e^{-sH})_{s\geq 0 }$ is a strongly continuous
semigroup on $\mathcal L (L^2)$ so: $$\forall f \in L^2 \virg  \lim \limits_{s \to
0} \|e^{-sH}f-f\|_{L^2} = 0.$$ \\
Let $f \in H^1 \cap L^2$ so that $T f \in L^2$ and let $a$ be an atom. We also have 
$$| \langle T e^{-sH}f - T f \virg a \rangle |\leq \|e^{-sH} Tf - Tf\|_{L^2} \|a\|_{L^2}
\underset{s\to 0}{\to}0.$$
Consequently, uniformly with respect to the atom $a$, we have
 $$| \langle Tf \virg a \rangle |= \lim \limits_{s \to 0} | \langle T e^{-sH}f \virg a \rangle | \lesssim A
\|f\|_{H^1}.$$
Then for all $f\in H^1\cap L^2$: $$\|Tf\|_{\BMO} \lesssim A \|f\|_{H^1}.$$

As $\BMO$ is a Banach space, $T$ admits an extension (still denoted $T$) which is
bounded from $H^1$ to $\BMO$ and then from $H^1$ to $(H^1)^*$ because $\BMO
\hookrightarrow (H^1)^*$.
\end{proof}

 \subsection{Interpolation}

Having obtained a bound on the Hardy space, we now aim to use interpolation to
conclude the proof of Theorem \ref{thm1}.

\begin{proof}[Proof of Theorem \ref{thm1}]

Consider a $L^2$-bounded operator $T$, which commutes with $H$ and satisfies
Property \eqref{disp} for some $m\in[ \frac{d}{2},\frac{4M}{3}]$.
Then Theorem \ref{thm_hardy_1} shows that $T$ admits a continuous extension from
$H^1$ to $(H^1)^*$. So we aim now to interpolate the two following continuities: 
$$\begin{cases}
 &\|T\|_{L^2 \to L^2} \lesssim 1 \\
 &\|T\|_{H^1 \to (H^1)^*} \lesssim A. \\
\end{cases}$$
Let $p$ be fixed in $(1,2)$. Then by choosing $\theta=\frac 2p-1 \in (0,1)$ and
$\frac 1q = 1- \frac 1p$, that is $q=p'$, in Theorem \ref{interpolation}, if
$\mu(X)=+\infty$, we have
$$T \colon (L^2, H^1)_{\theta} =L^p\to (L^2, (H^1)^*)_{\theta}\hookrightarrow L^q.$$
It follows the boundedness of $T$ from $L^p$ to $L^{p'}$. More precisely, if the
space $X$ is unbounded then
$$\|T\|_{L^p \to L^{p'}} \lesssim \|T\|_{H^1 \to (H^1)^*}^{\theta} \|T\|_{L^2 \to
L^2}^{1-\theta} \lesssim A^ \theta  = A^{\frac 1p - \frac 1{p'}}.$$
If the space $X$ is bounded, then Theorem \ref{interpolation_finie} shows
$$\|T\|_{L^p \to L^{p'}} \lesssim A^{\frac 1p - \frac 1{p'}}+B,$$
provided that $\|T\|_{L^p \to L^2} \lesssim B$.
\end{proof}

\section{Application to Strichartz estimates} \label{section_strichartz}

In this section we aim to take advantage of the dispersive estimates previously
obtained in the particular situation where $T$ is given by the Schr\"odinger
propagator, to deduce some Strichartz estimates with loss of derivatives, as
introduced in \cite{BGT}.

In particular, we are looking to dispersive estimates $L^p-L^{p'}$ with polynomial
bound. It is also natural to work in the setting of an Ahlfors regular space (and not
only in the doubling situation). The space $X$ of homogeneous type is said Ahlfors regular if there exist two
absolute positive constants $c$ and $C$ such that for all $x \in X$ and $r > 0$:
\begin{equation}\label{ah}
 c r^{d} \leq \mu(B(x, r)) \leq C r^{d}.
\end{equation}
From now on, we will assume this property.

To establish Strichartz estimates from dispersive inequalities we adapt the result
by Keel-Tao in \cite{KT}, namely: \\
Consider $(U(t))_{t \in {\mathbb R}}$ a collection of uniformly $L^2$-bounded
operators, i.e.
\begin{equation}\label{KT_1}
\sup_{t\in {\mathbb R}} \ \|U(t)\|_{L^2 \to L^2} \lesssim 1
\end{equation}
and such that for a certain $\sigma > 0$ 
\begin{equation}\label{KT_2}
\forall t\neq s \in \mathbb R, \|U(s)U(t)^*\|_{L^1 \to L^{\infty}} \lesssim |t-s|^{-\sigma}.
\end{equation}
Then in \cite{KT}, it is proved that for all admissible pair of exponents \eqref{admissible}, we have 
$$\|U(t)f\|_{L^p_t L^q_x} \lesssim \|f\|_{L^2}.$$

By the exact same proof, we have the following
\begin{thm}\label{thm_KT}
Suppose that the collection $(U(t))_t$ satisfies \eqref{KT_1} and for some $\sigma >0$
\begin{equation}\label{KT_3}
\forall t\neq s \in \mathbb R, \|U(s)U(t)^*\|_{H^1 \to (H^1)^*} \lesssim |t-s|^{-\sigma}.
\end{equation}
Then for all admissible pair \eqref{admissible} with $q\neq +\infty$, we have
 $$\|U(t)f\|_{L^p_t L^q_x} \lesssim \|f\|_{L^2},$$
where we assume in addition that 
\begin{equation}\label{KT_4}
\forall t\neq s \in \mathbb R, \|U(s)U(t)^*\|_{L^{q'} \to L^2} \lesssim |t-s|^{-\sigma\left( \frac{1}{q'} - \frac{1}{q} \right)}
\end{equation}
if $X$ is bounded.
\end{thm}

We do not give a proof of this result, since it is exactly the same as the one in
\cite{KT} by replacing the space $L^1$ with the Hardy space $H^1$. 
The proof relies on interpolating the two boundedness \eqref{KT_1} and \eqref{KT_3},
which is still possible with the Hardy space, due to Theorem \ref{interpolation}.

We are now in position to prove the following result:

\begin{thm}\label{thm_strichartz_2-bis} Assume \eqref{ah} with \eqref{due}. Consider an integer $\ell\geq 0$. Assume that the operator $T_t(H):=e^{itH} \psi_{2\ell} (h^2H)$ satisfies
Property $(H_m(|t|^{-\frac d2}))$ for some $m\geq \frac d2$ and every $t\in[-1,1]$.
Then for all pair of admissible exponents \eqref{admissible} with $q\neq +\infty$ we have: 
$$\left(\int_{-1}^1 \|e^{itH}\psi_{2\ell}(h^2H)f\|_{L^q}^pdt \right)^{\frac 1p}\lesssim \|\psi_{\ell,\frac{1}{2}}(h^2H)f\|_{L^2}.$$
\end{thm}

We also have the ``semi-classical'' version, involving a loss of derivatives: 

\begin{thm}\label{thm_strichartz_2} Assume \eqref{ah} with \eqref{due}. Consider an integer $\ell\geq 0$. Assume that for some $h_0>0$ and
$\gamma\in[0,2)$ (or $\gamma \in [1,2)$ if $X$ is bounded) the operator $T_t(H):=e^{itH} \psi_{2\ell} (h^2H)$ satisfies
Property $(H_m(|t|^{-\frac d2}))$ for some $m\geq \frac d2$ and every $t$
satisfying 
$$ |t|\lesssim h^\gamma \qquad \textrm{and} \quad h\leq h_0.$$
Then for all pair of admissible exponents \eqref{admissible} with $q\neq +\infty$ we have 
$$\left(\int_{-1}^1 \|e^{itH}\psi_{2\ell}(h^2H)f\|_{L^q}^pdt \right)^{\frac 1p}\lesssim h^{-\frac{\gamma}{p}} \|\psi_{\ell,\frac{1}{2}}(h^2H)f\|_{L^2}.$$
\end{thm}

\begin{rem} 
\begin{enumerate}
 \item Following the arguments of Proposition \ref{prop_hyp_2}, if $e^{itH}
\psi_{2\ell } (h^2H)$ satisfies Property $(H_m(|t|^{-\frac d2}))$ for some integer
$\ell \geq 0$ then $e^{itH} \psi_{2\ell'} (h^2H)$ also satisfies Property
$(H_m(|t|^{-\frac d2}))$ for every integer $\ell'\geq \ell$.
 \item The case $\gamma \geq 2$ is easy (as explained in the Introduction). When $X$ is bounded, one cannot expect $\gamma=0$ because of the example of a constant initial data (see Introduction).
\end{enumerate}
\end{rem}

\begin{proof}[Proof of Theorems \ref{thm_strichartz_2-bis} and \ref{thm_strichartz_2} ] 
We only detail the proof of Theorem \ref{thm_strichartz_2} which is slightly more technical, we let the minor modifications to the reader to prove Theorem \ref{thm_strichartz_2-bis}.

Fix an interval $J\subset [-1,1]$ of length $|J|=h^\gamma$ and consider
$$U(t)=\mathbf1_Je^{itH}\psi_{\ell,\frac{1}{2}}(h^2H)$$
We aim to apply Theorem \ref{thm_KT} with $\sigma = \frac d2$ and a suitable large
enough integer $M$ (defining the Hardy space). So fix this integer $M\geq \frac{3m}{4}$ large enough which allows
us to consider atoms and Hardy space, related to this parameter and we have
$m\in[\frac d2,\frac{4M}{3}]$ as required in  Theorem \ref{thm_hardy_1}.

Since $x\mapsto e^{itx}\psi_{\ell,\frac{1}{2}}(x) \in L^{\infty}(\mathbb R_+)$ is
uniformly bounded, with respect to $t$, then Theorem \ref{thm_calcul_fonctionnel}
yields that 
$$ \sup_{t>0} \
\|U(t)f\|_{L^2}=\|\mathbf1_Je^{itH}\psi_{\ell,\frac{1}{2}}(h^2H)f\|_{L^2} \lesssim
\|f\|_{L^2},$$
which is \eqref{KT_1}.

Then let us check \eqref{KT_3}. We have
\begin{align*}
U(t)U(s)^*
&=\mathbf1_J(t)\mathbf1_J(s)e^{itH}\psi_{\ell,\frac{1}{2}}(h^2H)(e^{isH}\psi_{\ell,\frac{1}{2}}(h^2H))^*
\\
                &=\mathbf1_J(t) \mathbf1_J(s) T_{t-s}(H),
\end{align*}
where we used that $H$ is self-adjoint and $|\psi_{\ell,\frac{1}{2}}|^2=\psi_{2\ell}$.
Since $J$ has a length equal to $h^\gamma$ then $U(t)U(s)^*$ is vanishing or
else $|t-s|\leq h^\gamma$. In this last case, $U(t)U(s)^*$ 
satisfies Property $(H_{m}(|t-s|^{-\frac d2}))$. Hence, by Theorem
\ref{thm_hardy_1}, we deduce that
$$\|U(t)U(s)^*f\|_{(H^1)^*}\lesssim \frac 1{|t-s|^{\frac d2}}\|f\|_{H^1},$$
which is \eqref{KT_3}.
Let us check \eqref{KT_4} in case $X$ is bounded: similarly since the Schr\"odinger propagators are unitary in $L^2$, we have 
$$\|U(t)U(s)^*\|_{L^{q'} \to L^2} \leq \|\psi_{2\ell}(h^2H)\|_{L^{q'} \to L^2}$$
with $|t-s|\leq h^{\gamma}\lesssim 1$.
Recall that 
for all $q'\in [1,2)$: 
$$ \|\psi_{4\ell,2}(h^2H)f\|_{L^{q}} \lesssim h^{-d\left( \frac{1}{q'} - \frac{1}{q} \right)}\|f\|_{L^{q'}}.$$
By a $TT^*$ argument we have: $$\|\psi_{4\ell,2}(h^2H)\|_{L^{q'} \to L^{q}}=\|\psi_{2\ell}(h^2H)\|_{L^{q'} \to L^2}^2.$$
Hence $$\|U(t)U(s)^*\|_{L^{q'} \to L^2} \lesssim h^{-\frac d2 \left( \frac{1}{q'} - \frac{1}{q} \right)}\leq |t-s|^{-\frac{d}{2\gamma}\left( \frac{1}{q'} - \frac{1}{q} \right)} \lesssim |t-s|^{-\frac{d}{2}\left( \frac{1}{q'} - \frac{1}{q} \right)}$$
as soon as $\gamma \geq 1$.

Thus we can apply Theorem \ref{thm_KT}. For all admissible pair \eqref{admissible} with $q\neq+ \infty$, then 
$$\left(\int_{\mathbb R} \|U(t)g\|_{L^q}^pdt\right)^{\frac 1p} \lesssim \|g\|_{L^2}.$$
That is 
$$\left(\int_{J} \|e^{itH}\psi_{\ell,\frac{1}{2}}(h^2H)g\|_{L^q}^pdt\right)^{\frac
1p} \lesssim \|g\|_{L^2}.$$
Take $g=\psi_{\ell,\frac{1}{2}}(h^2H) f$ then
$\psi_{\ell,\frac{1}{2}}(h^2H)g=\psi_{2\ell}(h^2H) f$ and so
\begin{equation}\label{strichartz_J}
\left(\int_{J} \|e^{itH}\psi_{2\ell}(h^2H)f\|_{L^q}^pdt\right)^{\frac 1p} \lesssim
\| \psi_{\ell,\frac{1}{2}}(h^2H) f \|_{L^2}.
\end{equation}

We write $[-1,1]=\displaystyle \bigcup_{k=1}^N J_k,$ where $J_k$ are disjoint
intervals with a length smaller than $h^{\gamma}$, so the number of intervals
satisfies $N \lesssim h^{-\gamma}$. 

Hence, by \eqref{strichartz_J}
 $$\int_{-1}^1 \|e^{itH}\psi_{2\ell}(h^2H)f\|_{L^q}^pdt \lesssim \displaystyle
\sum_{k=1}^N \int_{J_k} \|e^{itH}\psi_{2\ell}(h^2H)f\|_{L^q}^pdt \lesssim N
\|\psi_{\ell,\frac{1}{2}}(h^2H) f\|_{L^2}^p,$$
and so
 $$\left(\int_{-1}^1 \|e^{itH}\psi_{2\ell}(h^2H)f\|_{L^q}^pdt \right)^{\frac 1p}
\lesssim h^{-\frac {\gamma}{p}} \|\psi_{\ell,\frac{1}{2}}(h^2H) f\|_{L^2}.$$
\end{proof}

We can now prove the main result of this section: How Property $(H_m(t^{-\frac d2}))$ implies
Strichartz estimates with loss of $\frac{\gamma}{p}$ derivatives: 
\begin{thm}\label{thm_strichartz}
Assume \eqref{ah} with \eqref{due}. Consider an integer $\ell\geq 0$.
Assume that for some $h_0>0$ and $\gamma\in[0,2)$ the operator $T_t(H):=e^{itH}
\psi_{2\ell} (h^2H)$ satisfies Property $(H_m(|t|^{-\frac d2}))$ for some $m\geq \frac d2$ and every $t$ satisfying 
$$ |t|\lesssim h^\gamma \qquad \textrm{and} \quad h\leq h_0.$$
Then for all pair of admissible exponents \eqref{admissible} with $q\neq +\infty$, every solution $u=e^{itH}u_0$ of the
problem
 $$\begin{cases}
  &i\partial_tu+Hu=0 \\
  &u_{|t=0}=u_0\\
 \end{cases}$$ satisfies
 $$\|u\|_{L^p([-1,1], L^q)}\lesssim \|u_0\|_{W^{\frac{\gamma}{p}, 2}}.$$
\end{thm}

\begin{rem} We can consider more regular initial data, in the sense that if for some
$\delta>0$
$$\frac 2p + \frac dq = \frac d2-\delta,$$
then we have
$$\|u\|_{L^p([-1,1], L^q)}\lesssim \|u_0\|_{W^{\delta + \frac{\gamma}{p}, 2}}.$$
\end{rem}

\begin{proof}
Apply Theorem \ref{thm_LP} to $u(t)=e^{itH}u_0$
$$\|u(t)\|_{L^q}\lesssim \| \varphi(H)u(t)\|_{L^q}+\left\|\left(\int_0^{h_0}
|\psi_{2\ell}(s^2H)u(t)|^2\frac{ds}s\right)^{\frac 12}\right\|_{L^q}.$$
The function $\varphi$ is also depending of the parameters $h_0,\ell$. We omit this
dependence. Take the $L^p([-1,1])$ norm in time of that expression. Minkowski
inequality leads to
$$\|u\|_{L^p([-1,1],L^q)}\lesssim \underbrace{\|
\varphi(H)u\|_{L^p([-1,1],L^q)}}_{=I}+\underbrace{\left \|\left(
\int_0^{h_0}\|\psi_{2\ell}(s^2H)u\|_{L^q}^2\frac{ds}s\right)^{\frac12}\right\|_{L^p}}_{=II}.$$
Then \eqref{UE} with \eqref{ah} yields that $\varphi(H)$ has a kernel satisfying Gaussian pointwise estimate \eqref{UEmn} at the scale $1$ (or more precisely $h_0$ but we forget this
dependence) so is in particular bounded from $L^2$ to $L^q$ (since $q\geq 2$)
and so 
$$I\lesssim \|e^{itH}u_0\|_{L^p([-1,1], L^2)} \lesssim \|u_0\|_{L^2} \lesssim
\|u_0\|_{W^{\frac{\gamma}{p},2}},$$
because the Schr\"odinger group is an isometry on $L^2$.

Since $p\geq 2$, generalized Minkowski inequality and Theorem \ref{thm_strichartz_2}
yield
\begin{align*}
II&\leq \left(\int_0^{h_0} \|\psi_{2\ell}(s^2H)u\|_{L^p([-1,1],L^q)}^2\,
\frac{ds}s\right)^{\frac 12} \\
   &\lesssim \left(\int_0^{h_0} s^{-\frac{2\gamma}{p}}
\|\psi_{\ell,\frac{1}{2}}(s^2H)u_0\|_{L^2}^2 \, \frac{ds}s\right)^{\frac 12} \\
   &\lesssim \left\| \left( \int_0^{h_0} s^{-\frac{2\gamma}{p}}
|\psi_{\ell,\frac{1}{2}}(s^2H)u_0|^2 \, \frac{ds}s\right)^{\frac
12}\right\|_{L^2} \lesssim \|u_0\|_{W^{\frac{\gamma}{p},2}},
\end{align*}
where we used $\ell>\frac{\gamma}{p}$ (since $\ell\geq 1$, $\gamma\in[0,2)$ and
$p\geq 2$) and the fact that 
$$s^{-\frac{\gamma}{p}}\psi_{\ell,\frac{1}{2}}(s^2H) =
\psi_{\ell-\frac{\gamma}{2p},\frac{1}{2}}(s^2H) H^{\frac{\gamma}{2p}}$$
with Theorem \ref{thm_LP}.
Finally, we get
$$\|u\|_{L^p([-1,1],L^q)}\lesssim \|u_0\|_{W^{\frac{\gamma}{p},2}}.$$
\end{proof}

\section{Dispersive estimates for Schr\"odinger operator through wave operator}\label{section_wave_propagation}

\subsection{Dispersive estimates from Wave to Schr\"odinger propagators}

We recall that we want to obtain $$\|T_t(H)\|_{H^1 \to \BMO} \lesssim |t|^{-\frac
d2}$$ where $T_t(H)=e^{itH} \psi_{2\ell}(h^2H)$ for $t$ belonging to an interval, as large as possible.
In regard of the previous section, it suffices to check that
$e^{itH}\psi_{2\ell}(h^2H)$  satisfies Property $(H_m(|t|^{-\frac d2}))$ (for some
parameters $\ell,m,\gamma,h_0$), which may be written with \eqref{ah} as: for
every balls $B_r,\widetilde{B_r}$ 
\begin{equation}\label{hm(t-d2)}
\|e^{itH}\psi_{2\ell}(h^2H) \psi_m(r^2H)\|_{L^2(B_r)\to L^2(\widetilde{B_r})} \lesssim
\left(\frac{r^2}{|t|}\right)^{\frac d2}.
\end{equation}

We aim to use the Hadamard formula, which describes how the Schr\"odinger propagator
may be built using the wave propagator. Let us quickly recall it: the Cauchy
formula gives that for any $a \in \mathbb C$ with $\Re(a) >0$
 $$a^{-\frac 12} e^{-\frac{\xi^2}{2a}}=(2\pi)^{-\frac{1}{2}} \int_{\mathbb R}
e^{-ix\xi}e^{-\frac{ax^2}{2}}dx.$$
Using imparity and noting $z=\frac{1}{2a}$, we get
 $$e^{-z\xi^2}=\frac{1}{\sqrt{\pi}}\int_0^{+\infty}\cos(s\xi)e^{-\frac{s^2}{4z}}\frac{ds}{\sqrt
z}.$$
Since $H$ is a self-adjoint nonnegative operator admitting a $L^\infty$-functional
calculus, one deduces the Hadamard transmutation formula:
\begin{equation}\label{astuce}
 e^{-zH}=\frac{1}{\sqrt{\pi}}\int_0^{+\infty}\cos(s\sqrt H)e^{-\frac{s^2}{4z}} \,
\frac{ds}{\sqrt z}.
\end{equation}


We now give a suitable condition on the wave propagators, under which
\eqref{hm(t-d2)} can be proved through \eqref{astuce}. The next section aims to
check that this assumption is satisfied in well-known situations as Euclidean space
or smooth Riemannian manifolds.

\begin{hyp}\label{hyp_cossH} There exists $\kappa\in(0,\infty]$ and an integer $m_0$
such that for every $s\in(0,\kappa)$  we have: for every $r>0$, every balls
$B_r,\widetilde{B_r}$ of radius $r$ then
$$\|\cos(s\sqrt H) \psi_{m_0}(r^2H)\|_{L^2(B_r) \to L^2(\widetilde{B_r})} \lesssim
\left(\frac{r}{s+r}\right)^{\frac{d-1}{2}} \left( 1+\frac{|L-s|}{r}
\right)^{-\frac{d+1}{2}}$$
 where $L=d(B_r,\widetilde{B_r})$.
 \end{hyp}

\begin{rem}
 Using the same arguments as in Proposition \ref{prop_hyp_2}, one can show that if
Assumption \ref{hyp_cossH} is true for an integer $m_0$ then it also holds for
every integer $m\geq m_0$.
\end{rem}

The main result of this section is the following:  

\begin{thm}\label{demo_hyp_Hmn2-bis} Suppose \eqref{ah} with $d>1$, \eqref{due} and
Assumption \ref{hyp_cossH} with $\kappa=\infty$. Then for every integer $m\geq \max(\frac d2, m_0+\left\lceil\frac{d-1}{2}\right\rceil)$ (where the
integer $m_0$ is the one given by Assumption \ref{hyp_cossH}) we have for every
$t\in {\mathbb R}$
 \begin{equation}\label{hyp_1-bis} \|e^{itH} \psi_{m}(r^2H)\|_{L^2(B_r) \to
L^2(\widetilde{B_r})} \lesssim \left( \frac{r^2}{|t|} \right)^{\frac d2},
 \end{equation}
 where the implicit constant only depends on integers $m,m'$.
Consequently, $e^{itH} $ satisfies Property $(H_m(|t|^{-\frac d2}))$
for every $t\in {\mathbb R}$.
\end{thm}

\begin{thm}\label{demo_hyp_Hmn2}
 Suppose \eqref{ah} with $d>1$, \eqref{due} and Assumption \ref{hyp_cossH} with
$\kappa\in(0,\infty)$. Then for every $\varepsilon>0$, every $h>0$ with
$|t|<h^{1+\varepsilon}$, and for every integers $m'\geq 0$ and $m\geq \max(\frac d2,
m_0+\left\lceil\frac{d-1}{2}\right\rceil)$ (where the integer $m_0$ is the one
given by Assumption \ref{hyp_cossH}) we have
 \begin{equation}\label{hyp_1} 
  \|e^{itH} \psi_{m'}(h^2H)\psi_{m}(r^2H)\|_{L^2(B_r) \to L^2(\widetilde{B_r})} \lesssim \left( \frac{r^2}{|t|} \right)^{\frac d2},
 \end{equation}
where the implicit constant only depends on $\varepsilon>0$ and integers $m$, $m'$.
Consequently, $e^{itH} \psi_{m'}(h^2H)$ satisfies Property $(H_m(|t|^{-\frac d2}))$
for every $|t|<h^{1+\varepsilon}$ and every $\varepsilon>0$.
\end{thm}

\begin{proof}[Proof of Theorems \ref{demo_hyp_Hmn2-bis} and \ref{demo_hyp_Hmn2} ]

We only prove Theorem \ref{demo_hyp_Hmn2}, which is more difficult and let the
reader to check that the exact same proof allows us to get Theorem
\ref{demo_hyp_Hmn2-bis}, which is indeed easier since the quantity $I_\kappa$
(defined later in the proof) is vanishing.

\smallskip
\noindent
\underline{Step 1}: Some easy reductions.

Remark that the case $r\geq \sqrt{|t|}$ is easy via the bounded functional calculus,
indeed 
$$\|e^{itH}\psi_{m'}(h^2H) \psi_{m}(r^2H)\|_{L^2(B_r)\to L^2(\widetilde{B_r})}\leq
\|e^{it\cdotp}\psi_{m'}(h^2\cdotp)\psi_{m}(r^2\cdotp)\|_{L^{\infty}(\mathbb
R_+)}\lesssim 1 \lesssim \left( \frac{r^2}{|t|} \right)^{\frac d2}.$$
So now we only restrict our attention and assume that $r^2\leq |t|$.

Then assume that \eqref{hyp_1} is proved for every $h\in(0,r]$.  We aim to check
that it also holds for $h>r$. So fix balls of radius $r<h$. It comes
\begin{align*}
 \|e^{itH}& \psi_{m'}(h^2H) \psi_{m}(r^2H)\|_{L^2(B_r)\to L^2(\widetilde{B_r})} \\
 &\lesssim \frac{r^{2m}}{(\frac{h^2}{2}+r^2)^{m}}\|e^{itH} \psi_{m'}(\frac{h^2}{2}H)
\psi_{m}((\frac{h^2}{2}+r^2)H)\|_{L^2(B_r)\to L^2(\widetilde{B_r})}\\
 &\lesssim \left( \frac{r}{h} \right)^{2m} \|e^{itH}  \psi_{m'}(\frac{h^2}{2}H) 
\psi_{m}((\frac{h^2}{2}+r^2)H) \|_{L^2(B_\rho)\to L^2(\widetilde{B_\rho})}, 
\end{align*} 
 where $\rho=\sqrt{\frac{h^2}{2}+r^2} \geq r$, $\rho\simeq h$ and we write
$B_{\rho}=\frac{\rho}{r}B_r$ the dilated ball (similar notation for
$\widetilde{B_\rho}$).
Using \eqref{hyp_1} at the scale $\rho$ (since $\rho \geq h/\sqrt{2}$) yields 
\begin{align*}
 \|e^{itH} \psi_{m'}(h^2H) \psi_{m}(r^2H)\|_{L^2(B_r)\to L^2(\widetilde{B_r})} &
\lesssim \left( \frac{r}{h} \right)^{2m} \left( \frac{\rho^2}{|t|} \right)^{\frac d2}
\\
  & \lesssim \left( \frac{r}{h} \right)^{2m} \left( \frac{h^2}{|t|} \right)^{\frac d2} \\
  & \lesssim \left( \frac{r^2}{|t|} \right)^{\frac d2},
\end{align*}
where we have used that $m\geq d/2$ and (since $r\leq h$) 
$$\frac{r^{2m}}{h^{2m}}h^d = r^d \frac{r^{2m-d}}{h^{2m-d}}\leq r^d.$$  
So as soon as \eqref{hyp_1} will be proved for $h\leq r$, then the other case
immediately follows.

Consequently, we can restrict our study to $h\leq r$ and $r^2 \leq |t|$, that we now
assume for the sequel.

For an integer $m' \neq 0$, we have
$$e^{itH}\psi_{m'}(h^2H) \psi_{m}(r^2H)=  \left( \frac{h^2}{r^2} \right)^{m'}
e^{itH} e^{-h^2H} \psi_{m'+m}(r^2H).$$
Using $h\leq r$, it comes  
$$ \|e^{itH}\psi_{m'}(h^2H)\psi_{m}(r^2H)\|_{L^2(B_r) \to L^2(\widetilde{B_r})}
\lesssim \|e^{(it-h^2)H}\psi_{m'+m}(r^2H)\|_{L^2(B_r) \to L^2(\widetilde{B_r})}.$$

So if \eqref{hyp_1} is proved for $m'=0$ and some integer $m$ then by Theorem
\ref{thm:H}, it also holds for $m'=0$ and any integer $m'' \geq m$. Hence, by the
previous observation, \eqref{hyp_1} will hold for every $m'=m''-m\geq0$.

Finally, we can restrict our attention to prove \eqref{hyp_1} for $m'=0$ with $h\leq
r$ and $r^2 \leq |t|$, which is now supposed for the rest of the proof. 

\smallskip
\noindent
\underline{Step 2}: Decomposition into three regimes.

We fix the parameter $h$ and consider $e^{itH}e^{-h^2H}=e^{-zH}$ with $z=h^2-it$.
By the representation \eqref{astuce}, it comes
$$e^{-zH} = \int_0^{\infty} \cos(s\sqrt H)e^{-\frac{s^2}{4z}}\frac{ds}{\sqrt {\pi
z}}.$$
We split this integral into three ranges. Let us consider a smooth cut-off function
$\chi \in C^{\infty}(\mathbb R_+)$ such that 
$\begin{cases}
 &0\leq \chi \leq 1\\
 &\chi(x)=1 \text{ if } x \in [0, \frac{|t|}{r}]\\
 &\chi(x)=0 \text{ if } x \in [\frac{2|t|}{r}, +\infty]
\end{cases}$,
with $\forall n \in \mathbb N$, $\|\chi^{(n)}\|_{L^\infty} \lesssim\left(\frac{
r}{|t|}\right)^n$. 

We split the integral into three terms 
$$e^{-zH}= \int \chi(z) \cos(s\sqrt H)e^{-\frac{s^2}{4z}}\frac{ds}{\sqrt{\pi z}} +
\int_{\frac{|t|}{r}}^{\kappa} (1-\chi(z))\cos(s\sqrt H)e^{-\frac{s^2}{4z}}\frac{ds}{\sqrt
{\pi z}} + I_\kappa(H),$$
where $I_\kappa=0$ if $\kappa=\infty$ and else
$$I_\kappa(H):= \int_{\kappa}^\infty (1-\chi(z))\cos(s\sqrt
H)e^{-\frac{s^2}{4z}}\frac{ds}{\sqrt {\pi z}}.$$

\smallskip
\noindent
\underline{Step 3}: The two last regimes.

The second term is estimated using Assumption \ref{hyp_cossH} as follows (we
recall that $z=h^2-it$ so that $|z|\simeq |t|$)
\begin{align*}
 &\left \| \int_{\frac{|t|}{r}}^{\kappa} (1-\chi(z)) \cos(s\sqrt H) \psi_{m}(r^2H)
e^{-\frac{s^2}{4z}} \frac{ds}{\sqrt z} \right \|_{L^2(B_r) \to L^2(\widetilde
{B_r})} \\
 &\lesssim \int_{\frac{|t|}{r}}^{\kappa} \left( \frac rs \right)^{\frac{d-1}{2}} \left(
1+\frac{|L-s|}{r} \right)^{-\frac{d+1}{2}} \frac{ds}{\sqrt{|t|}} \\
 &\lesssim \int_{\frac{|t|}{r}}^{\kappa} \left( \frac{r}{\frac{|t|}{r}}
\right)^{\frac{d-1}{2}} \left( 1+\frac{|L-s|}{r} \right)^{-\frac{d+1}{2}}
\frac{ds}{\sqrt{|t|}} \\
 &\lesssim \int_{0}^{\infty} \left( \frac{r^2}{|t|} \right)^{\frac{d-1}{2}} \left( 1+u
\right)^{-\frac{d+1}{2}} \frac{rdu}{\sqrt{|t|}} \\
 &\lesssim \left( \frac{r^2}{|t|} \right)^{\frac d2}.
\end{align*}

The last term $I_\kappa(H)$ is estimated by only using the $L^2$-boundedness of the
wave propagator:
\begin{align*}
\| I_\kappa(H)\psi_m(r^2H) \|_{L^2(B_r) \to L^2(\widetilde {B_r})} & \lesssim 
\int_{\kappa}^{+\infty} \|\cos(s\sqrt H)\psi_{m}(r^2H)\|_{L^2(B_r) \to
L^2(\widetilde {B_r})} e^{-\frac{s^2}{4}\Re(\frac{1}{z})} \frac{ds}{\sqrt{|z|}} \\
& \lesssim  \int_{\kappa \sqrt{\Re(\frac{1}{4z}})}^{+\infty}
e^{-u^2}\frac{du}{\sqrt{\Re(\frac{1}{z})}\sqrt{|z|}} \\
&\lesssim  \left(\int_0^{+\infty} e^{-\frac{u^2}{2}} du\right) e^{-\frac{\kappa^ 2
\Re(\frac{1}{4z})}{2}} \left( \sqrt{\Re\left(\frac1z\right)} \sqrt{|t|} \right)^{-1}.
\end{align*}
Given that $\Re(\frac{1}{z})=\frac{h^2}{h^4+t^2} \gtrsim \frac{h^2}{t^2}$ (since we assumed $|t|\geq h^2$, see Step 1), we get 
$$ \| I_\kappa(H) \psi_m(r^2H)\|_{L^2(B_r) \to L^2(\widetilde{B_r})} \lesssim
\frac{|t|^{\frac 12}}{h}\left( \frac{h}{|t|} \right)^{-k},$$ 
for $k>0$ as large as we want because $\frac{h}{|t|} \gtrsim 1$ (indeed $|t|\leq
h^{1+\varepsilon} \lesssim h$). Note that the implicit constant here may depend on
$\kappa$.
Since we have reduced the situation to $h \leq r$, it comes 
$$\frac{|t|^{\frac 12}}{h}\left( \frac{|t|}{h} \right)^k \lesssim \left( \frac{h}{\sqrt{|t|}}
\right)^d \lesssim \left(\frac{r}{\sqrt{|t|}}\right)^d$$ 
as soon as $|t|^{\frac 12 + k + \frac d2} \leq h^{1+k+d}$, i.e. $|t|\leq
h^{\frac{1+k+d}{\frac 12 + k + \frac d2}}\leq h^{1+\frac{\frac 12 + \frac d2}{\frac
12 + k + \frac d2}}$ which is true for $k$ large enough since $|t|\leq
h^{1+\varepsilon}$.

So we have obtained the desired bound for the two last terms. It remains to study
the first and more difficult one.

\smallskip
\noindent
\underline{Step 4}: The first regime.

We aim to use integration by parts in $s$. For all integer $n\geq 0$, all $s>0$ and
$\Re(z)>0$, we have 
$$\partial_s^n \left( e^{-\frac{s^2}{4z}} \right) = e^{-\frac{s^2}{4z}} \left(
c_n\frac{s^{n}}{z^{n}} + c_{n-1}\frac{s^{n-2}}{z^{n-1}}+ \ldots + c_{n-2
\lfloor\frac n2\rfloor}\frac{s^{n-2 \lfloor\frac n2\rfloor}}{z^{n-\lfloor\frac
n2\rfloor}} \right),$$
where $(c_j)_j$ are numerical constants.
Making $2n$ integrations by parts gives (as soon as $m\geq n$) 
\begin{align*}
 &\int_0^{\infty} \cos(s\sqrt H) \psi_m(r^2 H)\chi(s) e^{-\frac{s^2}{4z}} ds \\
 =&\int_0^{\infty} \frac{\cos(s\sqrt H)}{H^n}  \psi_m(r^2 H) \partial_s^{2n}
\left[\chi(s) e^{-\frac{s^2}{4z}} \right] ds \\
 =& \int_0^{\infty} {\cos(s\sqrt H)} r^{2n}\psi_{m-n}(r^2 H)\sum_{k=0}^{2n} c_k
\chi^{(2n-k)}(s) \partial_s^k\left( e^{-\frac{s^2}{4z}} \right)ds \\
 =& \int_0^{\infty} \cos(s\sqrt H) r^{2n}\psi_{m-n}(r^2 H) \sum_{k=0}^{2n}
\chi^{(2n-k)}(s) e^{-\frac{s^2}{4z}} \left( c_k\frac{s^{k}}{z^{k}} + \ldots +
c_{n-2 \lfloor\frac n2\rfloor}\frac{s^{k-2 \lfloor\frac
k2\rfloor}}{z^{k-\lfloor\frac k2\rfloor}} \right)ds,
\end{align*}
where $c_j$ always denotes some numerical constants, possibly changing from line to
line.
The behaviour of the sum over $k$ is governed by its two extremal terms (that is $k=0$ and $k=2n$
where we only keep the first and last terms of the sum) which leads us to (since
$|z|\simeq |t|$)
\begin{align*}
&\left \| \int_0^{+\infty} \cos(s\sqrt H)\psi_m(r^2H) \chi(s) e^{-\frac{s^2}{4z}}
\frac{ds}{\sqrt z} \right \|_{L^2(B_r) \to L^2(\widetilde{B_r})} \\
 \lesssim& \int_0^{2\frac{|t|}{r}} \|\cos(s\sqrt H) \psi_{m-n}(r^2H)\|_{L^2(B_r) \to
L^2(\widetilde{B_r})} r^{2n} \left[ \left(\frac{r}{|t|} \right)^{2n} + \left(\frac{s}{|t|}
\right)^{2n} + \frac{1}{|t|^n} \right] \frac{ds}{\sqrt{|t|}} \\
 \lesssim& \int_0^{2\frac{|t|}r} \left(\frac{r}{r+s} \right)^{\frac{d-1}{2}} \left( 1+
\frac{|L-s|}{r} \right)^{-\frac{d+1}{2}} r^{2n} \left[ \left(\frac{r}{|t|} \right)^{2n}
+ \left(\frac{s}{|t|} \right)^{2n} + \frac{1}{|t|^n} \right] \frac{ds}{\sqrt{|t|}},
\end{align*}
where we used Assumption \ref{hyp_cossH} (this is allowed if $m-n\geq m_0$) and
$L:=d(B_r,\widetilde{B_r})$.
If $n=\lceil\frac{d-1}{2}\rceil$, then firstly 
\begin{align*}
 \int_0^{2\frac{|t|}{r}} \left(\frac{r}{s+r} \right)^{\frac{d-1}{2}} \left( 1+
\frac{|L-s|}{r} \right)^{-\frac{d+1}{2}} \left( \frac{r^2}{|t|}\right)^{2n}
\frac{ds}{\sqrt{|t|}} &\leq \int_0^{+\infty} (1+u)^{-\frac{d+1}{2}} \frac{rdu}{\sqrt{|
t|}} \left(\frac{r^2}{|t|}\right)^{2n} \\
 &\lesssim \left(\frac{r^2}{|t|}\right)^{2n+\frac12} \leq
\left(\frac{r^2}{|t|}\right)^{\frac d2},
\end{align*}
since $d>1$ and  $2n+\frac 12 \geq \frac d2$.
For the second term, we have
\begin{align*}
 \int_0^{\frac{|t|}{r}}& \left(\frac{r}{r+s} \right)^{\frac{d-1}{2}} \left( 1+
\frac{|L-s|}{r} \right)^{-\frac{d+1}{2}} \left( \frac{rs}{|t|}\right)^{2n}
\frac{ds}{\sqrt{|t|}} \\
 & \leq  \frac{r^{\frac{d-1}{2}+2n}}{|t|^{2n}} \int_0^{\frac{|t|}{r}} \left( 1+
\frac{|L-s|}{r} \right)^{-\frac{d+1}{2}} s^{2n-\frac{d-1}{2}} \frac{ds}{\sqrt{|t|}} \\
 & \lesssim \frac{r^{\frac{d-1}{2}+2n}}{|t|^{2n}} \left( \frac{|t|}{r}
\right)^{2n-\frac{d-1}{2}} \int_0^{+\infty}(1+u)^{-\frac{d+1}{2}}\frac{rdu}{\sqrt{|
t|}}\lesssim \left(\frac{r^{2}}{|t|}\right)^{\frac d2},
\end{align*}
since $2n-\frac{d-1}{2} \geq 0$.
And for the third and last term, it comes
\begin{align*}
 \int_0^{\frac{|t|}{r}} \left(\frac{r}{r+s} \right)^{\frac{d-1}{2}} \left( 1+
\frac{|L-s|}{r} \right)^{-\frac{d+1}{2}} \frac{r^{2n}}{|t|^n} \frac{ds}{\sqrt{|t|}}
&\leq \left( \frac{r^2}{|t|} \right)^n\int_{0}^{+\infty} (1+u)^{-\frac{d+1}{2}}
\frac{rdu}{\sqrt{|t|}}\\
 &\lesssim \left( \frac{r^2}{|t|} \right)^{n+\frac 12}\leq \left( \frac{r^2}{|t|}
\right)^{\frac d2},
\end{align*}
since $n+\frac 12 \geq \frac d2$.
The intermediate terms in the integrations by parts have an intermediate behaviour.
We point out that these last computations required $m-n\geq m_0$ which is true,
since $m \geq m_0+\left \lceil\frac{d-1}{2}\right\rceil$ and $n=\lceil \frac{d-1}{2}
\rceil$.

That concludes the proof, since each of the three terms have a satisfying bound.
\end{proof}

\subsection{A digression about these dispersive properties and the spectral measure}

Let us assume Assumption \ref{hyp_cossH} for $\kappa=1$.

Following the same reasoning as in Sections \ref{section_dispersion} and \ref{section_strichartz}, it
comes that the assumed inequality
$$\|\cos(s\sqrt H) \psi_{m_0}(r^2H)\|_{L^2(B) \to L^2(\widetilde B)} \lesssim
\left(\frac{r}{s+r}\right)^{\frac{d-1}{2}} \left( 1+\frac{|L-s|}{r}
\right)^{-\frac{d+1}{2}}$$
allows us to prove that $\cos(s \sqrt{H})$ is bounded from the Hardy space $H^1$ to
$\BMO$ (built with some parameter $M$ sufficiently large) with
$$ \| \cos(s \sqrt{H})\psi_{1}(r^2H) \|_{H^1 \to \BMO} \lesssim r^{-\frac{d+1}{2}}
(s+r)^{-\frac{d-1}{2}}, \quad \forall |s|\leq 1.$$
That corresponds to the $H^1\to \BMO$ counterpart of more classical $L^1\to
L^\infty$ dispersive estimates. Following interpolation and Keel-Tao's argument (as
detailed previously) for the wave propagator, it allows us to deduce Strichartz
estimates for the wave equations:
for exponents $p,q$ {\it wave-admissible} and $\delta\geq 0$ satisfying 
$$\frac 1p + \frac dq = \frac d2-\delta,$$
every solution $u=\cos(t\sqrt{H})u_0$ of the problem
 $$\begin{cases}
  &\partial^2_{tt}u+Hu=0 \\
  &u_{|t=0}=u_0\\
  & \partial_t u_{|t=0}=0
 \end{cases}$$ satisfies: 
 \begin{equation} \|u\|_{L^p([-1,1], L^q)}\lesssim \|u_0\|_{W^{\delta, 2}}.
\label{stri} \end{equation}

Such Strichartz estimates for the wave equation, allow us to deduce some sharp
$L^2-L^q$ estimates for the spectral projector (introduced by Sogge \cite{Sogge}),
as detailed by Smith in \cite{Smith}. Without details, we just sketch the proof of
\cite{Smith} to check that it can be adapted to this very general setting.

 Indeed, consider $\lambda>0$ and the spectral projector
$$ \Pi_\lambda = {\bf 1}_{[\lambda, \lambda +1)}(\sqrt{H}).$$
Define the function
$$ \rho_\lambda(x):=\int_{-1}^{1} e^{-it \lambda} \cos(t x) dt$$
which a direct computation gives
$$ \rho_\lambda(x) = \frac{\sin(\lambda-x)}{\lambda -x} + 
\frac{\sin(\lambda+x)}{\lambda +x}.$$
So we observe that $\rho_\lambda(x)\in[\frac{1}{2},2]$ if $x\in[\lambda,\lambda+1)$.
As a consequence, by bounded $L^2$-functional calculus, we deduce that for $f\in L^2$
$$ \Pi_\lambda(f) = \int_{-1}^{1} e^{-it \lambda} \cos(t \sqrt{H})
\left[\rho_\lambda(\sqrt{H})^{-1} \Pi_\lambda f\right] dt,$$
with $\rho_\lambda(\sqrt{H})^{-1} \Pi_\lambda$ a uniformly $L^2$-bounded operator
(and also in any $L^2$ Sobolev space since it commutes with $H$).

By applying \eqref{stri}, we deduce that for $q\in[\frac{2(d+1)}{d-1},\infty)$ 
\begin{align*}
 \|\Pi_\lambda(f)\|_{L^q}  & \lesssim \left\|\cos(t \sqrt{H})
\left[\rho_\lambda(\sqrt{H})^{-1} \Pi_\lambda f\right] \right\|_{L^2([-1,1],L^q)} \\
  & \lesssim \left\| \rho_\lambda(\sqrt{H})^{-1} \Pi_\lambda f  \right\|_{W^{\delta(q), 2}} \\
   & \lesssim \left\| \Pi_\lambda f  \right\|_{W^{\delta(q), 2}} \lesssim \lambda^
{\delta(q)} \|f\|_{L^2},
\end{align*}
where $\delta(q)$ is given by
$$\frac 12 + \frac dq = \frac d2-\delta(q).$$

By interpolating with the trivial $L^2-L^2$ bound, we deduce (as explained in
\cite{Smith}) that
\begin{equation} \| \Pi_\lambda \|_{L^2 \to L^q} \lesssim \left\{ \begin{array}{l}
\lambda^{\frac{d-1}{2}\left(\frac{1}{2}-\frac{1}{q}\right)}, \quad \textrm{if $2\leq
q \leq 2\frac{d+1}{d-1}$} \\
\lambda^{d\left(\frac{1}{2}-\frac{1}{q}\right)-\frac{1}{2}}, \quad \textrm{if $q
\geq 2\frac{d+1}{d-1}$} \\
\end{array}
\right. \label{es:1}
\end{equation}

Let us point out that if now we assume Assumption \ref{hyp_cossH} for
$\kappa=\infty$, then by combining Theorems \ref{thm_strichartz} and
\ref{demo_hyp_Hmn2-bis}  we get free dispersive estimates without loss of
derivatives: for $p\in(1,2]$ then
$$ \|e^{it H}\|_{L^p \to L^{p'}} \lesssim |t|^{-\frac{d}{2}\left(\frac{1}{p} -
\frac{1}{p'} \right)}$$
uniformly  with respect to $t\in {\mathbb R}$. Then if the operator $H$ (or $\sqrt{H}$) has
a spectral measure with a Radon-Nicodym derivative, then following \cite[Corollary
3.3]{BO}, we know that {\it Restriction estimates} hold which are:
$$  \left\| \frac{dE_H(\lambda)}{d\lambda} \right\|_{L^p \to L^{p'}} \lesssim
\lambda^{\frac{d}{2}(\frac{1}{p} - \frac{1}{p'}) -1},$$ 
where $E_H(\lambda)$ is the spectral measure of $H$ and $p\in[1,\frac{2d}{d+2})$. We
also have other estimates for higher order derivatives and we refer to \cite{BO} for
more details.
Such estimates give in particular for $\lambda\geq 1$
\begin{align}
  \|\Pi_\lambda\|_{L^p \to L^{p'}}  & \lesssim \int_{\lambda^2}^{(\lambda +1)^2}
\left\|\frac{dE_H(s)}{ds} \right\|_{L^p \to L^{p'}} ds \nonumber \\
  & \lesssim \int_{\lambda^2}^{(\lambda +1)^2} s^{\frac{d}{2}(\frac{1}{p} -
\frac{1}{p'}) } \frac{ds}{s} \nonumber \\
  & \lesssim  \lambda^{d(\frac{1}{p} - \frac{1}{p'})- 1 } \lesssim
\lambda^{2d(\frac{1}{p} - \frac{1}{2})- 1 }. \label{es:2}
  \end{align}
We then exactly recover the estimate in \eqref{es:1} but the range for $q=p'$ in
\eqref{es:1} is larger than the one obtained by \eqref{es:2}: indeed the range in
\eqref{es:1} is given by the sharp critical exponent $1\leq p \leq 2
\frac{d+1}{d+3}$.

\section{The Euclidean and Riemannian cases}\label{section_euclidean_case}

To enhance the legitimacy of Assumption \ref{hyp_cossH}, we check its validity  for
the Laplace-Beltrami operator $H=-\Delta$ in four situations:
\begin{itemize}
\item the Euclidean space $X=\mathbb R^d$ with $\kappa=\infty$;
\item any smooth compact Riemannian manifold of dimension $d$ and $\kappa$ is given
by the injectivity radius;
\item any smooth noncompact Riemannian manifold of dimension $d$, with
$C^\infty_b$-geometry and $\kappa$ given by the geometry;
\item Smooth perturbation of the Euclidean space $X=\mathbb R^d$, $H= -\frac{1}{\rho}
\nabla \cdot(A \nabla \cdot)$ (for uniformly nondegenerate function $\rho$ and
matrix $A$, with bounded derivatives) which is a self-adjoint operator on $\R^d$,
equipped with the measure $d\mu = \rho dx$, with $\kappa<\infty$ (given  by $A$ and
$\rho$).
\end{itemize}

\begin{prop} \label{prop:ass} In these four previous cases, Assumption
\ref{hyp_cossH} is satisfied.
\end{prop}

The proof is based on the following properties (which are a refinement of the finite
speed propagation property): for $B, \widetilde B$ two balls of radius $r$, then with
$L=d(B,\widetilde B)$ and $s\in(0,\kappa)$:
\begin{itemize}
\item If $L>s+2r$ then the finite speed propagation property occurs
 \begin{equation} \|\cos(s\sqrt H)\|_{L^2(B) \to L^2(\widetilde B)}=0; \label{eq:fspp}
 \end{equation}
\item if $L\leq s-10r$ then 
\begin{equation}
\|\cos(s\sqrt H)\|_{L^2(B) \to L^2(\widetilde B)} \lesssim \left( \frac{r}{r+s}
\right)^{\frac{d-1}{2}} \left(1+\frac{|L-s|}{r}\right)^{-\frac{d+1}{2}}.
\label{eq:fspp2}
\end{equation}
\end{itemize}

We refer the reader to the introduction for more details about the finite speed
propagation property, which yields \eqref{eq:fspp}. Property \eqref{eq:fspp2} is
quite standard, see for example \cite{Berard} for the case of a compact Riemannian
manifold (where a short time parametrix is detailed) and appendix \ref{section_appendice_1} where we detail computations in the Euclidean situation.

In particular, we partly recover the results of \cite{BGT,ST} (up to a loss
$\varepsilon>0$ as small as we want). Indeed, by combining Proposition \ref{prop:ass}
with Theorems \ref{thm_strichartz} and \ref{demo_hyp_Hmn2} (with
$\gamma=1+\varepsilon$), we have the following:

\begin{corollaire} Any smooth compact Riemannian manifold or non-compact Riemannian
manifold with a $C^\infty_b$ geometry (or as previously for a smooth perturbation of
the Euclidean setting with suitable functions $\rho,A$) satisfy Strichartz estimates
with a loss of derivatives $\frac{1}{p}+\varepsilon$, for every $\varepsilon>0$.
\end{corollaire}

As a conclusion, we have obtained that as soon as we have suitable (short time) $L^2-L^2$ microlocalized dispersive properties on the wave equation 
(Assumption \ref{hyp_cossH}) then we can obtain their Strichartz estimates and dispersive estimates for Schr\"odinger equation (with an eventual 
loss of derivatives if $\kappa<\infty$). We just point out that in the case of a convex subset of the Euclidean space with a boundary, 
then wave operators for the Dirichlet Laplacian do not satisfy Assumption \ref{hyp_cossH} (since there is a loss of $1/4$ in the main 
exponent), see \cite{ILP} by Ivanovici, Lebeau and Planchon.

\begin{proof}[Proof of Proposition \ref{prop:ass}]

We detail the proof in the Euclidean case with $\kappa=\infty$. We let the reader to
check that everything still holds (up to some change of notations) for a compact
Riemannian manifold with $\kappa$ given by the injectivity radius. Indeed, the proof
relies on \eqref{eq:fspp2} and a precise formulation of the wave kernel around the
light cone, which is obtained by the Hadamard parametrix (and has the same form as
in the Euclidean case), see \cite{Berard}. So let us focus on the Euclidean
situation.

First, if $s\leq 10 r$ then by the finite speed propagation property and
Davies-Gaffney estimates, we have
\begin{align*} 
\|\cos(s\sqrt H) \psi(r^2H)\|_{L^2(B) \to L^2(\widetilde B)} & \leq
\|\psi(r^2H)\|_{L^2(B) \to L^2(10 \widetilde B)} \\
 & \lesssim e^{-\frac{d(B,10 \widetilde B)^2}{4r^2}} \lesssim
\left(1+\frac{d(B,\widetilde B)}{r}\right)^{-\frac{d+1}{2}}
=\left(1+\frac{L-s+s}{r}\right)^{-\frac{d+1}{2}}\\
 & \lesssim
\left(\frac{r}{r+s}\right)^{\frac{d-1}{2}}\left(1+\frac{|L-s|}{r}\right)^{-\frac{d+1}{2}},
\end{align*}
since $s\lesssim r$, which is the desired estimate.

So we now only focus in the situation where $s\geq 10 r$ and consider $(B_k)_k$ a
bounded covering of $X$, by balls of radius $r$. Let $\chi_{B_k}$ be a smooth
partition of the unity, adapted to this covering: so $\chi_{B_k}$ is supported in
$2B_k$, takes values in $[0,1]$ and satisfies for all $n \in \mathbb N$ 
\begin{equation}\label{regularity}
 \|\nabla^n \chi_{B_k}\|_{L^{\infty}} \leq \frac{1}{r^n}.
\end{equation}

We decompose 
$$ \|\cos(s\sqrt H) \psi(r^2H)\|_{L^2(B) \to L^2(\widetilde B)} \leq \sum_{B_k}
\|\cos(s\sqrt H)(\chi_{B_k}. \psi(r^2H))\|_{L^2(B) \to L^2(\widetilde B)}.$$
Due to \eqref{eq:fspp}, the sum is restricted to balls $B_k$ such that
$d(B_k,\widetilde B)\leq s+2r$.

\smallskip
\noindent
\underline{Step 1}: The case $d(B_k,\widetilde B)\leq s-10r$.

Using \eqref{eq:fspp2} and Davies-Gaffney estimates, it comes 
\begin{align*}
 &\sum_{d(B_k,\widetilde B)\leq s-10r} \|\cos(s\sqrt
H)(\chi_{B_k}.\psi(r^2H))\|_{L^2(B) \to L^2(\widetilde B)}\\
 \leq& \sum_{d(B_k,\widetilde B)\leq s-10r} \|\cos(s\sqrt H)\|_{L^2(B_k) \to
L^2(\widetilde B)}\|\psi(r^2H)\|_{L^2(B) \to L^2(B_k)} \\
 \lesssim& \sum_{d(B_k,\widetilde B)\leq s-10r} \left(\frac{r}{r+s}
\right)^{\frac{d-1}{2}} \left( 1+ \frac{s-d(B_k,\widetilde B)}{r}
\right)^{-\frac{d+1}{2}} e^{-\frac{d(B,B_k)^2}{4r^2}}.
\end{align*}
Note that $s-d(B_k,\widetilde B) \geq 10r\geq 0$. 

 We can evaluate the following sum
\begin{align} \label{eq:sum}
 \sum_{k}e^{-\frac{d(B,B_k)^2}{4r^2}}&\leq \sum_{l=0}^{+\infty} e^{-2^{2l}} \sharp
\{k, \frac{d(B,B_k)}{2r} \sim 2^l \} \lesssim \sum_{l=0}^{+\infty}
2^{ld}e^{-2^{2l}}<+\infty.
\end{align}
We distinguish two cases. If $s-d(\widetilde B,B_k)\geq \frac 12 |s-d(B,\widetilde
B)|=\frac{1}{2}|s-L|$ then
\begin{align*}
 \sum_{d(B_k,\widetilde B)\leq s-10r}& \left(\frac rs \right)^{\frac{d-1}{2}}
\left(1+ \frac{s-d(B_k,\widetilde B)}{r}
\right)^{-\frac{d+1}{2}}e^{-\frac{d(B,B_k)^2}{4r^2}} \\
 &\leq \sum_k \left(\frac rs \right)^{\frac{d-1}{2}} \left(1+ \frac{|s-L|}{2r}
\right)^{-\frac{d+1}{2}}e^{-\frac{d(B,B_k)^2}{4r^2}}\\
 &\lesssim \left(\frac rs \right)^{\frac{d-1}{2}} \left(1+ \frac{|s-L|}{r}
\right)^{-\frac{d+1}{2}}.
\end{align*}
If $s-d(\widetilde B,B_k)\leq \frac 12 |s-d(B,\widetilde B)|$ then $$d(B,B_k) \geq
|d(B_k,\widetilde B)-d(B,\widetilde B)|=|(s-L)-(s-d(\widetilde B,B_k))| \geq \frac12
|s-L|.$$
Hence
\begin{align*}
 \sum_{d(B_k,\widetilde B)\leq s-10r}& \left(\frac rs \right)^{\frac{d-1}{2}}
\left(1+ \frac{s-d(B_k,\widetilde B)}{r}
\right)^{-\frac{d+1}{2}}e^{-\frac{d(B,B_k)^2}{4r^2}} \\
 &\lesssim \left(\frac rs \right)^{\frac{d-1}{2}} \sum_{d(B_k,\widetilde B)\leq
s-10r} \left( 1+\frac{10r}{r}
\right)^{-\frac{d+1}{2}}e^{-\frac{d(B,B_k)^2}{8r^2}}\underbrace{e^{-\frac{d(B,B_k)^2}{8r^2}}}_{\leq
e^{-\frac{|s-L|^2}{16r^2}}} \\
 &\lesssim \left(\frac rs \right)^{\frac{d-1}{2}} \left(1+ \frac{|s-L|}{r}
\right)^{-\frac{d+1}{2}}
\end{align*}
because for every $x\geq 0$, $e^{-x^2} \lesssim \left( 1+x \right)^{-\alpha}$ for
all $\alpha>1$. 

\smallskip
\noindent
\underline{Step 2}: The case $s-10r \leq d(B_k,\widetilde B)\leq s+2r$ with an odd
dimension $d\geq 3$.

In this case, we have to use a sharp expression of the kernel of the wave
propagator. It is known that the behaviour of the kernel is different according to
the parity of the dimension.
Let us start with the case of an odd dimension $d\geq 3$.
In the Euclidean situation, we have an exact representation of the kernel (see
\cite{Fol} e.g.): for every $s\geq 0$ and every sufficiently smooth function $g$
\begin{align*}
 \cos(s\sqrt H)g(x)&= \partial_s\left(\frac 1s \partial_s\right)^{\frac{d-3}{2}}
\left( s^{d-2} \int_{|y|=1} g(x+sy)dy \right)\\
 &= \sum_{n=0}^{\frac{d-1}{2}} c_n s^n \int_{|y|=1}\partial^n_s (g(x+sy))dy,
\end{align*}
where $c_n$ are some numerical constants.

Consider $g= \chi_{B_k} \psi(r^2H)f$ then it satisfies the following regularity
estimates (with a slight abuse of notations): for every integer $n\geq 0$
$$|\partial_s^n(\chi_{B_k} (x+sy) \psi(r^2H)f(x+sy))|\lesssim \frac{1}{r^n}
\widetilde \chi_{B_k}(x+sy) \widetilde \psi(r^2H)f(x+sy).$$
Let us explain this point. Indeed, we can control the derivatives of $\chi_{B_k}$ by
\eqref{regularity}. It remains to explain the behaviour of the derivatives of
$\psi(r^2H)f(x+sy)$.
The kernel of the heat semigroup, for $t>0$, is $$p_t(x,y)=\frac{1}{(4\pi t)^{\frac
d2}}e^{-\frac{|x-y|^2}{4t}}.$$ 
Thus for all $r>0$: $$\partial_s(p_{r^2}(x+sy,z)) = \frac{1}{(4\pi r^2)^{\frac
d2}}e^{-\frac{|x+sy-z|^2}{4r^2}} \frac{(x+sy-z)y}{2r^2}.$$
Hence
\begin{align*}
 |\partial_s(p_{r^2}(x+sy,z)| &\lesssim \frac{1}{(4\pi r^2)^{\frac
d2}}\frac{1}{2r^2}e^{-\frac{|x-y|^2}{4r^2}}|x+sy-z|\\
 &=\frac 1r\frac{1}{(4\pi r^2)^{\frac
d2}}\frac{|x+sy-z|}{2r}e^{-\left(\frac{|x+sy-z|}{2r}\right)^2}\\
 &\lesssim \frac 1r \frac{1}{(4\pi r^2)^{\frac d2}} e^{-\frac{|x+sy-z|^2}{8r^2}},
\end{align*}
which means that, up to some numerical constants, the $n$th derivative of
$\psi(r^2H)f(x+sy)$ behaves as $\frac{1}{r^n} \psi(r^2H)f(x+sy)$ in the sense that
their kernels have both similar Gaussian pointwise decays. Such a property also
holds on a compact smooth Riemannian manifold.

So we have for $f\in L^2(B)$ a function supported on $B$, 
\begin{align*}
 \|\cos&(s\sqrt H)(\chi_{B_k}.\psi(r^2H)f)\|_{L^2(\widetilde B)} \\
 &\lesssim \sum_{n=0}^{\frac{d-1}{2}} \left( \frac sr \right)^n \int_{|y|=1}
\|\widetilde \chi_{B_k}(x+sy) \widetilde \psi(r^2H)f(\cdot+sy)\|_{L^2(\widetilde
B)} dy \\
 &\lesssim \sum_{n=0}^{\frac{d-1}{2}} \left( \frac sr \right)^n \int_{S(0,1)\cap A}
\left \| \widetilde \psi(r^2H)f \right \|_{L^2(B_k)} dy,
\end{align*}
where $S(0,1)$ is the unit sphere and $A=\frac 1s (B_k-\widetilde B)$.
Hence from the exponential decay of the kernel of $\widetilde \psi(r^2 H)$, we get
\begin{align*}
 \|\cos(s\sqrt H)(\chi_{B_k}.\psi(r^2H)f)\|_{L^2(\widetilde B)}  &\lesssim
\sum_{n=0}^{\frac{d-1}{2}} \left( \frac sr \right)^n |S(0,1)\cap A| e^{-c
\frac{d(B,B_k)^2}{r^2}} \|f\|_{L^2(B)} \\
 &\lesssim \|f\|_{L^2(B)} \sum_{n=0}^{\frac{d-1}{2}} \left( \frac rs \right)^{d-1-n}
e^{-c\frac{d(B,B_k)^2}{r^2}} \\
 &\lesssim \|f\|_{L^2(B)} \left( \frac rs \right)^{\frac{d-1}{2}}
e^{-c\frac{d(B,B_k)^2}{r^2}}
\end{align*}
where we have used that the $(d-1)$-dimensional volume of $S(0,1)\cap A=S(0,1) \cap
\frac 1s (B_k-\widetilde B)$ is equivalent to $\left(\frac{r}{s}\right)^{d-1}$ and
$\left( \frac rs \right)^{d-1-n} \leq \left( \frac rs \right)^{\frac{d-1}{2}}$.
Hence, it remains to evaluate the sum $$\sum_{s-10r\leq d(\widetilde B, B_k) \leq
s+2r} e^{-c\frac{d(B,B_k)^2}{r^2}}.$$
Since $$d(B,B_k) \geq |d(B,\widetilde B) - d(\widetilde B, B_k)|-2r \geq |L-s|-4r.$$
Then $$|L-s|^2 \leq 2(d(B,B_k)^2+16r^2)$$ that is $$d(B,B_k)^2 \geq
\frac{|L-s|^2}{2}-16r^2.$$
Thus, we deduce
\begin{align*}
 \sum_{\underset{s-10r\leq d(\widetilde B, B_k) \leq s+2r}{B_k}}
e^{-c\frac{d(B,B_k)^2}{r^2}} &\leq \sum_{B_k} e^{-c\frac{d(B,B_k)^2}{2r^2}}
e^{-c\frac{|L-s|^2}{2r^2}}\\
 &\lesssim \left(1+\frac{|L-s|}{r} \right)^{-\frac{d+1}{2}}.
\end{align*}
In the end, we have obtained that 
$$\sum_{s-10r\leq d(\widetilde B, B_k) \leq s+2r} \|\cos(s\sqrt
H)(\chi_{B_k}.\psi(r^2H)f)\|_{L^2(\widetilde B)} \lesssim \left( \frac rs
\right)^{\frac{d-1}{2}} \left(1+\frac{|L-s|}{r} \right)^{-\frac{d+1}{2}}
\|f\|_{L^2(B)}$$
which gives the desired estimate (for an odd dimension).

\smallskip
\noindent
\underline{Step 3}: The case $s-10r \leq d(B_k,\widetilde B)\leq s+2r$ with an even
dimension $d\geq2$.

In this case the wave propagator is given by
\begin{align*}
 \cos(s\sqrt H)g(x)&= \partial_s\left(\frac 1s \partial_s\right)^{\frac{d-2}{2}}
\left( s^{d-1} \int_{|y|<1} g(x+sy)\frac{dy}{\sqrt{1-|y|^2}} \right)\\
 &= \sum_{n=0}^{\frac{d}{2}} c_n s^n
\int_{|y|<1}\partial_s(g(x+sy))\frac{dy}{\sqrt{1-|y|^2}},
\end{align*}
with some numerical constants $c_n$.
The same arguments as above give $$\|\cos(s\sqrt H)
(\chi_{B_k}.\psi(r^2H)f)\|_{L^2(\widetilde B)} \lesssim
\|f\|_{L^2(B)}\sum_{n=0}^{\frac d2} \left( \frac sr \right)^n
e^{-c\frac{d(B_k,B)^2}{r^2}} \int_{y\in A \cap B(0,1)} \frac{dy}{\sqrt{1-|y|^2}},$$
where $A=\frac 1s (B_k-\widetilde B)$.
Moreover
\begin{align*}
 \int_{A \cap B(0,1)} \frac{dy}{\sqrt{1-|y|^2}} &\leq \int_{A\cap B(0,1-\frac
rs)}\frac{dy}{\sqrt{1-|y|^2}} +\int_{1-\frac rs}^1 |S(0,\rho) \cap A|
\frac{d\rho}{\sqrt{1-\rho^2}}\\
 &\leq \int_{A\cap B(0,1)} \left(\frac sr \right)^{\frac 12} dy+ \left(\frac rs
\right)^{d-1} \int_{1-\frac rs}^1 \frac{d\rho}{\sqrt{1-\rho}}\\
 &\lesssim \left(\frac rs \right)^{d-\frac 12}+\left(\frac rs \right)^{d-1}
\left[\sqrt{1-\rho}\right]_{1-\frac rs}^1 \lesssim \left(\frac rs \right)^{d-\frac
12}.
\end{align*}
Hence,
\begin{align*} 
 \sum_{s-10r\leq d(\widetilde B, B_k) \leq s+2r} \|\cos(s\sqrt H)
(\chi_{B_k}.\psi(r^2H)f)\|_{L^2(\widetilde B)} &\lesssim \sum_{B_k} \left( \frac rs
\right)^{\frac{d-1}{2}} e^{-c\frac{d(B,B_k)^2}{r^2}}\|f\|_{L^2(B)} \\
 &\lesssim \left( \frac rs \right)^{\frac{d-1}{2}} \left( 1+
\frac{|s-L|}{r}\right)^{-\frac{d+1}{2}} \|f\|_{L^2(B)},
\end{align*}
which gives the desired estimate.

Note that since $r\lesssim s$ we have $\frac rs \lesssim \frac{r}{r+s}$ so in any
dimension $d>1$:
 $$\|\cos(s\sqrt H)\psi(r^2H)\|_{L^2(B) \to L^2(\widetilde B)} \lesssim \left(
\frac{r}{r+s} \right)^{\frac{d-1}{2}}
\left(1+\frac{|L-s|}{r}\right)^{-\frac{d+1}{2}}.$$
\end{proof}


\appendix

\section{Wave propagation in the Euclidean setting}\label{section_appendice_1}

In this appendix, we aim to check \eqref{eq:fspp2} in the Euclidean situation, from
the exact and global formula giving the wave operators.
Let us consider the Euclidean space $X=\mathbb R^d$, equipped with its canonical
structure and $H=-\Delta$.
 
\begin{prop}\label{propagateur_s} For every balls $B_r, \widetilde{B_r}$ of radius
$r>0$ and every $s>0$, if $L:=d(B_r,\widetilde{B_r})  \leq s-10r$ then
 $$\|\cos(s\sqrt H)\|_{L^2(B_r) \to L^2(\widetilde{B_r})} \lesssim \left(
\frac{r}{r+s} \right)^{\frac{d-1}{2}} \left( 1+
\frac{|L-s|}{r}\right)^{-\frac{d+1}{2}}.$$
\end{prop}

\begin{proof}
 Let $f \in L^2(B_r)$.
 If $d\geq 3$ is odd then the wave propagator is given by
 $$\cos(s\sqrt H)f(x) = \sum_{n=0}^{\frac{d-1}{2}} c_ns^n
\int_{|y|=1}\partial_s(f(x+sy))\, dy,$$ for some numerical constants $c_n$.
 If $x\in \widetilde{B_r}$ and $x+sy \in B_r$ then $y=\frac{x+sy-x}{s} \in
\frac{B-\widetilde B}{s}$ hence
 \begin{equation}\label{support}
  |y| \leq \frac{d(B,\widetilde B) +2r}{s} \leq \frac{s-8r}{s} < 1.
 \end{equation}
 Thus $$\cos(s\sqrt H)f(x)=0.$$
 If $d \geq 2$ is even then the wave propagator is given by
  $$\cos(s\sqrt H)f(x)= \partial_s\left(\frac 1s \partial_s\right)^{\frac{d-2}{2}}
\left( s^{d-1} \int_{|y|<1} f(x+sy)\, \frac{dy}{\sqrt{1-|y|^2}} \right).$$
 Set $$I_{n,m}:=\int_{|y|<1} f(x+sy) \frac{|y|^{2m}}{(1-|y|^2)^n}\, dy.$$
 Since $$\cos(s\sqrt H)f(x)=\partial_s\left(\frac 1s
\partial_s\right)^{\frac{d-2}{2}} \left( s^{d-1} I_{\frac 12,0} \right)$$ we want
to evaluate 
 $$\partial_s I_{n,m}=\int_{|y|<1} \nabla f(x+sy).y \frac{|y|^{2m}}{(1-|y|^2)^n} \,
dy.$$
 By \eqref{support} the boundary term in Green's formula vanishes and so
$$\partial_s I_{n,m} = -\int_{|y|<1} \frac{f(x+sy)}{s}  \nabla\cdot
\left(\frac{y|y|^{2m}}{(1-|y|^2)^n}\right)\, dy.$$
 Consequently, it comes with numerical constants $\alpha_{n,m},\alpha_{n+1,m+1}$
  $$\partial_s I_{n,m} = \frac 1s (\alpha_{n,m}I_{n,m} +
\alpha_{n+1,m+1}I_{n+1,m+1}).$$
It follows that (with other coefficients but for simplicity we keep the same
notations) 
 $$\left( \frac 1s \partial_s \right) (s^{d-1} I_{n,m}) = s^{d-3}
(\alpha_{n,m}I_{n,m} + \alpha_{n+1,m+1}I_{n+1,m+1}).$$
 By iterating, we deduce that for $n=\frac{1}{2}$ and $m=0$
 $$\left( \frac 1s \partial_s \right)^{\frac{d-2}{2}}(s^{d-1} I_{\frac 12,0}) =
s^{d-1-(d-2)} (\alpha_{\frac 12,0}I_{\frac 12,0} + \cdots + \alpha_{\frac 12 +
\frac{d-2}{2},\frac{d-2}{2}}I_{\frac 12 + \frac{d-2}{2},\frac{d-2}{2}}).$$
 Hence,
  $$\cos(s\sqrt H) f(x) = \alpha_{\frac 12,0}I_{\frac 12,0} + \cdots +
\alpha_{\frac{d+1}{2},\frac d2}I_{\frac{d+1}{2},\frac d2},$$
where coefficients $\alpha_{n,m}$ are some numerical constants, possibly changing
from line to line.

 Since $\frac{1}{1-|y|^2} \geq 1$ and $|y|\leq 1$ we have: 
 \begin{align*}
  \|\cos(s\sqrt H)f\|_{L^2(\widetilde{B_r})} &  \lesssim \left \| \int_{|y|<1}
|f(x+sy)| \, \frac{dy}{(1-|y|^2)^{\frac{d+1}{2}}} \right \|_{L^2(\widetilde{B_r})}
\\
  &\lesssim \int_{B(0,1) \cap A} \|f(\cdot + sy)\|_{L^2(\widetilde{B_r})} \,
\frac{dy}{((1+|y|)(1-|y|))^{\frac{d+1}{2}}}
 \end{align*}
 where $A:=\frac 1s (B_r-\widetilde{B_r})$ so that $|y| \geq
\frac{d(B_r,\widetilde{B_r})-2r}{s}$. Moreover 
 $$\|f(\cdot +sy)\|_{L^2(\widetilde{B_r})} \leq \|f\|_{L^2(B_r)}.$$
 Hence: 
 \begin{align*}
  \|\cos(s\sqrt H)\|_{L^2(B) \to L^2(\widetilde B)} &\lesssim \frac{1}{\left(
1-\frac{L-2r}{s} \right)^{\frac{d+1}{2}}} |B(0,1) \cap A| \\
  &\lesssim \left( \frac{s-L+2r}{s} \right)^{-\frac{d+1}{2}} \left( \frac rs
\right)^d \lesssim \left( 1+ \frac{|s-L|}{r}\right)^{-\frac{d+1}{2}} \left( \frac
rs \right)^{\frac{d-1}{2}}\\
  &\lesssim \left( \frac{r}{r+s} \right)^{\frac{d-1}{2}} \left( 1+
\frac{|s-L|}{r}\right)^{-\frac{d+1}{2}},
 \end{align*}
where the last inequality holds if $r\leq s$. If $s\leq r$ then use $|B(0,1) \cap A|
\leq |B(0,1)| \lesssim 1$ to get the same estimation with $1\lesssim \frac{r}{r+s}$
instead of $\frac rs$.
\end{proof}


\small{

}
\begin{thebibliography}{99}

\bibitem{AP}
J-P. Anker and V. Pierfelice, 
\newblock Nonlinear Schr\"odinger equation on real hyperbolic spaces, 
\newblock {\it Ann. Henri Poincar\'e Analyse Non-Lin\'eaire} \textbf{26} (2009), no.
5, 1853--1869.

\bibitem{Ramona}
Anton, R. 
\newblock Strichartz inequalities for Lipschitz metrics on manifolds and the nonlinear Schr\"odinger equation on domains,
\newblock {\it Bull. Soc. Math. France} \textbf{136} (2008), no. 1, 27--65.

\bibitem{AT}
P. Auscher and P. Tchamitchian,
\newblock Square Root Problem for Divergence Operators and Related Topics, 
\newblock {\it Ast{\'e}risque} \textbf{249} (1998), Soc. Math. France

\bibitem{Aus1}
P. Auscher, 
\newblock  {On necessary and sufficient conditions for $L^p$ estimates of Riesz
transforms associated to elliptic operators on $\R^n$ and related estimates,}
\newblock {\it Memoirs of Amer. Math. Soc.} \textbf{186} no. 871 (2007).


\bibitem{ACDH}
P. Auscher, T. Coulhon, X. T. Duong and S. Hofmann, 
 Riesz transform on manifolds and heat kernel regularity,
 \textit{Ann. Sci. Ecole Norm. Sup.}, \textbf{37} (2004), no. 4, 911--957.

\bibitem{BBR}
N. Badr, F. Bernicot, E. Russ,
\newblock {Algebra properties for Sobolev spaces. Applications to semilinear
PDE's on manifolds,}
\newblock{\it  J. Anal. Math.} \textbf{118} (2012), 509--544.

\bibitem{BGX}
H. Bahouri, P. G\'erard and C-J. Xu,
\newblock {Espaces de Besov et estimations de Strichartz g\'en\'eralis\'ees sur le
groupe de Heisenberg,}
\newblock {\it J. Anal. Math.} \textbf{82} (2000), no. 1, 93--118.

\bibitem{Berard}
P. B\'erard,
\newblock {On the wave equation on a compact riemannian manifold without
conjugate points.}
\newblock {\it Math. Zeit.} \textbf{155} (1977), 249--276.

\bibitem{Ber}
F. Bernicot,
\newblock {Use of abstract Hardy spaces, real interpolation and applications to
bilinear operators,}
\newblock {\it Math. Zeit.} \textbf{265} (2010), 365--400.

\bibitem{BO}
F. Bernicot, El M. Ouhabaz, 
\newblock Restriction estimates via the derivatives of the heat semigroup and
connexion with dispersive estimates,
\newblock {\em Math. Res. Lett.} \textbf{20} (2013), no. 6, 1047--1058. 

\bibitem{BZ}
F. Bernicot, J. Zhao,
\newblock {New abstract Hardy spaces,}
\newblock {\it J. Func. Anal.} \textbf{255} (2008), 1761--1796.

\bibitem{BZ2}
F. Bernicot, J. Zhao, 
\newblock {Abstract framework for John Nirenberg inequalities and applications to Hardy spaces,}
\newblock {\it Ann. Sc. Norm. Super. Pisa Cl. Sci.} \textbf{11} (2012), no. 3, 475--501.

\bibitem{Blair}
M. Blair, 
\newblock Strichartz estimates for wave equations with coefficients of Sobolev regularity,
\newblock {\em Comm. Partial Differential Equations} \textbf{31} (2006), no. 4-6, 649--688.

\bibitem{BFHM}
M. D. Blair, G. A. Ford, S. Herr and J. L. Marzuola,
\newblock Strichartz Estimates for the Schr\"odinger Equation on Polygonal Domains,
\newblock {\it J. of Geom. Anal.} to appear.
DOI: 10.1007/s12220-010-9187-3

\bibitem{BSS}
M.D. Blair, H.F. Smith and C. D. Sogge,
\newblock On Strichartz estimates for Schr\"odinger operators in compact manifolds with boundary,
\newblock {\it Proc. Amer. Math. Soc.} \textbf{136} (2008), no. 1, 247--256.

\bibitem{BT}
 J.-M. Bouclet and N. Tzvetkov, 
 \newblock Strichartz estimates for long range perturbations,
\newblock {\it Amer. J. Math.} \textbf{129} (2007), no. 6, 1665--1609.

\bibitem{Bouclet}
 J.-M. Bouclet, 
 \newblock Strichartz estimates for asymptotically hyperbolic manifolds,
 \newblock {\it Analysis and PDE}, to appear.

\bibitem{Bourgain}
J. Bourgain;
\newblock Fourier transform restriction phenomena for certain lattice subsets and
application to nonlinear evolution equations I. Schr\"odinger equations,
\newblock {\it Geom. and Funct. Ana.} \textbf{3} (1993), 107--156.

\bibitem{BGT}
N. Burq, P. G\'erard, N. Tzvetkov,
\newblock{Strichartz inequalities and the nonlinear Schrodinger equation on
compact manifolds,}
\newblock {\it Amer. J. Math.} \textbf{126} (2004), no. 3, 569--605.

\bibitem{BGT2}
N. Burq, P. G\'erard and N. Tzvetkov,
\newblock On nonlinear Schr\"odinger equations in exterior domains,
\newblock {\it Ann. Inst. H. Poincar\'e Anal. Non Lin\'eaire} \textbf{21} (2004),
no. 3, 295--318. 

\bibitem{BGH}
N. Burq, C. Guillarmou and A. Hassell, 
\newblock Strichartz estimates without loss on manifolds with hyperbolic trapped
geodesics,
\newblock {\it Geom. Funct. Anal.} \textbf{20} (2010), no. 3, 627--656. 

\bibitem{Bow}
M. Bownik, 
\newblock {Boundedness of operators on Hardy spaces via atomic decompositions,}
\newblock {\it Proc. Amer. Math. Soc.} \textbf{133} (2005), 3535--3542.

\bibitem{CCO} G. Carron, T. Coulhon, E. M. Ouhabaz,
Gaussian estimates and Lp-boundedness of Riesz means,
\newblock {\it J. of Evol. Equ.}, \textbf{2} (2002) 299--317.

\bibitem{CW}
R. Coifman, G. Weiss, 
\newblock {Extensions of Hardy spaces and their use in analysis,}
\newblock {\it Bull. Amer. Math. Soc.} \textbf{83} (1977), 569--645.

\bibitem{CDL} T. Coulhon, X. T. Duong and X.D. Li,
Littlewood-Paley-Stein functions on complete Riemannian manifolds for $1<p<2$,
\newblock {\it Studia Math.}, \textbf{154} (2003), 37--57.

 \bibitem{CS} T. Coulhon and A. Sikora,
\newblock {Gaussian heat kernel bounds via Phragm\'en-Lindel\"of theorem,}  
\newblock {\it Proc. London Math. Soc.} \textbf{96} (2008),  507--544.

\bibitem{CRT}
T. Coulhon, E. Russ and V. Tardivel-Nachef,
\newblock Sobolev algebras on Lie groups and Riemannian manifolds, 
\newblock {\it Amer. J. of Math.} \textbf{123} (2001), 283--342.

\bibitem{Christianson}
H. Christianson, 
\newblock Dispersive estimates for manifolds with one trapped orbit,
\newblock {\it Comm. Partial Diff. Eq.} \textbf{33} (2008), no. 7-9, 1147--1174.

\bibitem{Davies}
E.B. Davies,
\newblock Non-Gaussian aspects of heat kernel behaviour, 
\newblock {\it J. London Math. Soc.} \textbf{55} (1997), 105--125.

\bibitem{DY1}
X.T. Duong, L. Yan, 
\newblock Duality of Hardy and BMO spaces associated with operators with heat kernel bounds,
\newblock {\it J. Amer. Math. Soc.} \textbf{18}, no.4 (2005), 943--973.

\bibitem{DY2}
X.T. Duong, L. Yan, 
\newblock New function spaces of BMO type, the John-Niremberg inequality, Interplation and Applications,
\newblock {\it Comm. on Pures and Appl. Math.} \textbf{58} no.10 (2005), 1375--1420.


\bibitem{FS}
C. Fefferman, E. Stein, 
\newblock{$H^p$ spaces of several variables,} 
\newblock {\it Acta Math.} \textbf{129} (1971), 137--193.

\bibitem{Fol}
G. B. Folland,
\newblock{\em Introduction to partial differential equations.}
\newblock Princeton University press (1976)

\bibitem{GV}
J. Ginibre and G. Velo,
\newblock Smoothing properties and retarded estimates for some dispersive evolution
equations,
\newblock {\it Comm. Math. Phys.} \textbf{123} (1989), 535--573.

\bibitem{Gr1} 
A.~Grigor'yan, 
\newblock Gaussian upper bounds for the heat kernel on arbitrary manifolds,
\newblock {\it J. Diff. Geom.}, \textbf{45} (1997), 33--52.

\bibitem{GSC}
P. Gyrya and L. Saloff-Coste,
\newblock {Neumann and Dirichlet heat kernels in inner uniform domains},
\newblock {\it Ast{\'e}risque} \textbf{33} (2011), Soc. Math. France.

\bibitem{HTW}
A. Hassell, T. Tao and J. Wunsch, 
\newblock Sharp Strichartz estimates on nontrapping asymptotically conic manifolds, 
\newblock {\it Amer. J. Math.} \textbf{128} (2006), no. 4, 963--1024.


\bibitem{HM}
S. Hofmann, S. Mayboroda, 
\newblock Hardy and BMO spaces associated to divergence form elliptic operators, 
\newblock {\it Math. Ann.} \textbf{344} (2009), no. 1, 37--116.

\bibitem{IP}
O. Ivanovici, F. Planchon, 
\newblock Square-function and heat flow estimates on domains,
\newblock {\it Comm. and PDE}.

\bibitem{ILP}
O. Ivanovici, G. Lebeau, F. Planchon, 
\newblock{Dispersion for the wave equation inside strictly convex domains I: The Friedlander model case.}
\newblock 

\bibitem{JN}
F. John, L. Nirenberg, 
\newblock {On functions of bounded mean oscillation,}
\newblock {\it Comm. Pure Appl. Math.} \textbf{14} (1961), 785--799.

\bibitem{KT}
M. Keel, T. Tao,
\newblock {Endpoint Strichartz estimates,}
\newblock {\it Amer. J. Math.} \textbf{120} (1998), no. 5, 955--980.

\bibitem{KU}
P.C. Kunstmann, M. Uhl,
\newblock  Spectral multiplier theorems of H\"ormander type on Hardy and Lebesgue spaces 
\newblock {\it submitted} (arXiv:1209.0694).

\bibitem{meda} S. Meda, On the Littlewood-Paley-Stein $g$-function, {\it
Trans. Amer. Math. Soc.} \textbf{347} (1995), 2201--2212.

\bibitem{MT}
J. Metcalfe and M. Taylor,
\newblock Nonlinear waves on 3D hyperbolic space,
\newblock {\it Trans. Amer. Math. Soc.} \textbf{363} (2011), 3489--3529.

\bibitem{RS}
M. Reed, B. Simon,
\newblock{\em Functional Analysis I.}
\newblock Methods of modern mathematical physics (1972)

\bibitem{RZ}
L. Robbiano and C. Zuily,
\newblock Strichartz estimates for Schr\"odinger equations with variable coefficients. 
\newblock {\it M\'em. Soc. Math. Fr.} \textbf{101-102} (2005). 

\bibitem{Robinson}
D. W. Robinson, 
\newblock {\it Elliptic Operators and Lie Groups},
\newblock Oxford Univerisity Press, 1991.


\bibitem{Smith}
H.F. Smith, 
\newblock Spectral cluster estimates for $C^{1,1}$ metrics,
\newblock {\em Amer. Jour. Math.} \textbf{128} (2006), 1069--1103. 

\bibitem{Sogge}
C. Sogge, 
\newblock Concerning the $L^p$ norm of spectral clusters for second order elliptic
operators on compact manifolds,
\newblock {\em  J. Funct. Anal.} \textbf{77} (1988), 123--134.

\bibitem{ST}
G. Staffilani, D. Tataru,
\newblock Strichartz estimates for a Schr\"odinger operator with nonsmooth
coefficients,
\newblock {\em Comm. Partial Differential Equations} \textbf{27} (2002), 1337--1372.

\bibitem{Ste}
E. Stein,
\newblock { \em Topics in Harmonic Analysis related to the Littlewood-Paley Theory.}
\newblock  Princeton University Press (1970)


\bibitem{Strichartz}
R. S. Strichartz, 
\newblock Restriction of Fourier Transform to Quadratic Surfaces and Decay of
Solutions of Wave Equations,
\newblock {\it Duke Math. J.} \textbf{44} (1977), 705--774.

\bibitem{TT}
H. Takaoka and N. Tzvetkov,
\newblock On 2D nonlinear Schr\"odinger equations with data on ${\mathbb R}\times\mathbb T$,
\newblock {\it  J. Funct. Anal.} \textbf{182} (2001),  no. 2, 427--442.

\bibitem{T}
D. Tataru,
\newblock Outgoing parametrices and global Strichartz estimates for Schr\"odinger
equations with variable coefficients. \\
\newblock Phase space analysis of partial differential equations, 291--313,
Progr. Nonlinear Differential Equations Appl., 69, Birkh\"auser Boston, Boston, MA,
2006. 

\bibitem{Taylor}
M. Taylor,
\newblock Hardy Spaces and Bmo on manifolds with bounded geometry,
\newblock {\it J. Geom. Anal.} \textbf{19} (2009), 137--190. 

\end{thebibliography}
\end{document}